\documentclass[12pt]{article}   

 \usepackage[T1]{fontenc}
      \input{dehyphtex}              
 \usepackage[width=16cm, height=22cm]{geometry}   
 \usepackage{amsmath,amssymb}         
 \usepackage[amsmath,thmmarks,hyperref]{ntheorem} 
 \usepackage{icomma}             
 \usepackage{enumerate}             
 \usepackage{mathrsfs}         
 \usepackage{slashed}
 \usepackage{color}
 
 \numberwithin{equation}{section}
 
 \usepackage[numbers,square]{natbib}
\theoremstyle{nonumberplain}  
\theoremheaderfont{\itshape}  
\theorembodyfont{\normalfont}  
\theoremseparator{.}  
\theoremsymbol{$\Box$}
\newtheorem{proof}{Proof} 
  
\theoremstyle{plain}  
\theoremheaderfont{\normalfont\bfseries}  
\theorembodyfont{\itshape}  
\theoremsymbol{~}
\theoremseparator{.}  
\newtheorem{proposition}{Proposition}[section]  
\newtheorem{corollary}[proposition]{Corollary}  
\newtheorem{lemma}[proposition]{Lemma}  
\newtheorem{theorem}[proposition]{Theorem}   

\theorembodyfont{\normalfont}  
 
\newtheorem{remark}[proposition]{Remark}
\newtheorem{example}[proposition]{Example}  
  
\newtheorem{definition}[proposition]{Definition} 

\theoremstyle{nonumberplain}
\theorembodyfont{\normalfont}

\newcommand{\R}{\mathbb{R}}
\newcommand{\e}{\mathrm{e}}
\newcommand{\V}{\mathcal{V}}
\newcommand{\N}{\mathbb{N}}

\newcommand{\C}{\mathbb{C}}
\newcommand{\dd}{\mathrm{d}}

\newcommand{\id}{\mathrm{id}}

\newcommand{\<}{\left\langle}
\renewcommand{\>}{\right\rangle}

\title{Strong Short Time Asymptotics and Convolution Approximation of the Heat Kernel}
\author{ Matthias Ludewig}

\begin{document}

\maketitle
 
\begin{center}
  Max-Planck Institute for Mathematics\\
 Vivatgasse 7 / 53119 Bonn \\ \medskip
 matthias$\_$ludewig@gmx.de
\end{center}

\begin{abstract}
We give a short proof of a strong version of the short time asymptotic expansion of heat kernels associated to Laplace type operators acting on sections of vector bundles over compact Riemannian manifolds, including exponential decay of the difference of the approximate heat kernel and the true heat kernel. We use this to show that repeated convolution of the approximate heat kernels can be used to approximate the heat kernel on all of $M$, which is related to expressing the heat kernel as a path integral. This scheme is then applied to obtain a short-time asymptotic expansion of the heat kernel at the cut locus.
\end{abstract}


\section{Introduction and Main Results}

Let $M$ be a compact Riemannian manifold of dimension $n$ and let $L$ be a Laplace type operator, acting on sections of a vector bundle $\V$ over $M$. For $t>0$, the heat kernel $p_t^L$ of $L$ is a smooth section of the bundle $\V \boxtimes \V^*$ over $M \times M$ (the vector bundle with fiber $\mathrm{Hom}(V_y, V_x)$ over the point $(x, y) \in M\times M$). It is well-known that for $x, y \in M$ close, the heat kernel has an asymptotic expansion of the form
\begin{equation} \label{HeatKernelExpansion}
  p_t^L(x, y) ~\sim~ \e_t(x, y) \sum_{j=0}^\infty t^j \frac{\Phi_j(x, y)}{j!},
\end{equation}
where
\begin{equation} \label{EuclideanHeatKernel}
  \e_t(x, y) = \frac{e^{-d(x, y)^2 / 4t}}{(4 \pi t)^{n/2}}
\end{equation}
is the Euclidean heat kernel (the name comes from the fact that $\e_t(x, y)$ is the heat kernel in case that $M = \R^n$ and $L=\Delta$, the usual Laplace operator). In \eqref{HeatKernelExpansion}, the ``correction terms'' $\Phi_j(x, y)$ are certain smooth sections of the bundle $\V \boxtimes \V^*$ over $M \bowtie M$, where $M \bowtie M = M \times M \setminus \{\text{cut points}\}$ is the set of points $(x, y) \in M \times M$ such that there is a unique minimizing geodesic connecting $x$ and $y$ (compare e.g.\ \cite[Section~2.5]{bgv}). In this paper, we will prove that the asymptotic relation \eqref{HeatKernelExpansion} can be made precise as follows.

\begin{theorem}[Strong Heat Kernel Asymptotics] \label{ThmStrongEstimates}
Let $L$ be a Laplace type operator, acting on sections of a vector bundle $\V$ over a compact Riemannian manifold $M$. Then for any compact subset $K$ of $M \bowtie M$, any $T>0$ and any numbers $\nu, k, l, m \in \N_0$, there exists a constant $C>0$ such that 
\begin{equation} \label{StrongAsymptoticsDifferential}
  \left| \frac{\partial^k}{\partial t^k}\nabla_x^l \nabla_y^m\left\{\frac{p_t^L(x, y)}{\e_t(x, y)} - \sum_{j=0}^\nu t^j \frac{\Phi_j(x, y)}{j!}\right\}\right| \leq C t^{\nu + 1-k}
\end{equation}
for all $(x, y) \in K$, whenever $0 < t \leq T$. Here $\Phi_j(x, y)$ are certain smooth sections of the bundle $\V\boxtimes \V^*$ over $M \bowtie M$.
\end{theorem}

In the theorem, $\nabla_x$ and $\nabla_y$ denote the covariant derivative with respect to the $x$ (respectively $y$) variable, where we use any metric connection on the bundle $\V$ (changing the connection only alters the constant $C$ on the right hand side).

\begin{corollary}
We have the complete asymptotic expansion
\begin{equation*}
  \frac{p_t^L(x, y)}{\e_t(x, y)} ~\sim~ \sum_{j=0}^\infty t^j \frac{\Phi_j(x, y)}{j!}
\end{equation*}
in the sense of topological vector spaces, in the Fréchet topology of $C^\infty(M \bowtie M , \V \boxtimes \V)$.
\end{corollary}

\medskip

Usually, the asymptotic relation \eqref{HeatKernelExpansion} is interpreted to say that for any $\nu \in \N_0$, 
\begin{equation} \label{WeakAsymptotics}
  \left| p_t^L(x, y) - \e_t(x, y) \sum_{j=0}^\nu t^j \frac{\Phi_j(x, y)}{j!} \right| \leq C t^{\nu +1}
\end{equation}
uniformly for $(x, y)$ over compact subsets of $M \bowtie M$. This statement is much weaker than Thm.~\ref{ThmStrongEstimates} (even in the case that $k=l=m=0$), since the latter implies that the right hand side of \eqref{WeakAsymptotics} can be replaced by $C t^{\nu+1} \e_t(x, y)$, which decays exponentially when $d(x, y) >0$. Proofs for the weaker statement can be found in various places in the literature (see \cite[Thm.~2.30]{bgv}, \cite[Thm.~7.15]{roe98}, \cite[3.2]{Rosenberg}, \cite[III.E]{BergerGauduchonMazet} just to name a few\footnote{Moreover, Chavel \cite[p.~154]{ChavelEigenvalues} claims to prove a version of the strong statement, but his proof is based on the wrong Lemma~1 on p.~152, which is incorrectly cited from \cite{BergerGauduchonMazet}.}).
The stronger result of Thm.~\ref{ThmStrongEstimates} seems to be somewhat folklore, but to the author's knowledge, no easily accessible proof exists in the literature outside either the theory of pseudo-differential operators, where one usually proves more general statements using a somewhat huge machinery (see e.g.\ \cite{GreinerHeatEquation} or \cite{MelroseAPS}), or the realm of stochastic analysis (e.g.~\cite{benarous}, \cite{azencott} or \cite{molchanov}).

The first goal of this paper is to give an easy proof of Thm.~\ref{ThmStrongEstimates} using the so-called {\em transmutation formula}, which relates the heat equation to the wave equation, and the Hadamard expansion of the wave kernel. This approach goes back to an older paper of Kannai \cite{Kannai}, who proves a variant of Thm.~\ref{ThmStrongEstimates} in the scalar case (compare also \cite{TaylorTwo}).

Thm.~\ref{ThmStrongEstimates} can be generalized to general complete manifolds. However, this is a somewhat intricate matter, as general Laplace Type operators need not have closed extensions generating operator semigroups. For formally self-adjoint Laplace type operators $L$, we prove that they have at most one such self-adjoint extension and that if they do, a version of Thm.~\ref{ThmStrongEstimates} holds for the corresponding heat kernel.

\medskip

The asymptotic expansion \eqref{HeatKernelExpansion} motivates to define {\em approximate heat kernels} $\e_t^\nu(x, y)$ by
\begin{equation} \label{ApproximateHeatKernel}
  \e_t^\nu(x, y) := \chi\bigl(d(x, y)\bigr) \,\e_t(x, y) \sum_{j=0}^\nu t^j \frac{\Phi_j(x, y)}{j!},
\end{equation}
where $\chi: [0, \infty) \longrightarrow [0, 1]$ is a smooth function  with $\chi(r) = 1$ near zero and support contained in $[0, \mathrm{inj}(M))$ (with $\mathrm{inj}(M)$ denoting the injectivity radius of $M$). If for general smooth kernels $k, \ell \in C^\infty(M \times M, \V \boxtimes \V^*)$, we define their convolution $k * \ell$ by
\begin{equation*}
  (k * \ell)(x, y) := \int_M k(x, z) \ell(z, y) \dd z, 
\end{equation*}
it turns out that the heat kernel $p_t^L(x, y)$ can be approximated by repeated convolutions of the kernel $\e_t^\nu(x, y)$. More precisely, we have the following result.

\begin{theorem}[Approximation by  Convolution] \label{ThmConvolutionApproximation}
~~Let $L$ be a formally self-adjoint Laplace type operator, acting on sections of a metric vector bundle $\V$ over a compact Riemannian manifold $M$. Then for any $\delta > 0$ with
\begin{equation}
  \delta < \left(\frac{\mathrm{inj}(M)}{\mathrm{diam}(M)}\right)^2,
\end{equation}
any $\nu \in \N_0$ and each $T>0$, there exists a constant $C>0$ such that 
\begin{equation}
  \left|p_t^L(x, y) - \bigl(\e_{\Delta_1\tau}^\nu * \cdots * \e_{\Delta_N \tau}^\nu\bigr)(x, y)\right| \leq C\,  p_t^\Delta(x, y)\,|\tau|^\nu \,t
 \end{equation}
for all $x, y \in M$ and for any partition $\tau = \{0 = \tau_0 < \tau_1 < \dots < \tau_N = t \leq T\}$ of an interval $[0, t]$  with $|\tau| \leq \delta t$, where $p_t^\Delta$ is the heat kernel of the Laplace-Beltrami operator on $M$. Here we used the notation $\Delta_j\tau := \tau_j - \tau_{j-1}$ and $|\tau| := \max_{1 \leq j \leq N} \Delta_j \tau$ for the increment, respectively the mesh of a partition $\tau$.
\end{theorem}

This approximation result can be used in different regimes: If one fixes $t>0$, one can make the $C^k$ difference in  between $p_t^L$ and $\e_{\Delta_1\tau}^\nu * \cdots * \e_{\Delta_N \tau}^\nu$ smaller than any given $\varepsilon>0$, by choosing a partition $\tau$ fine enough. On the other hand, by choosing $\nu$ large enough, this error can be made uniform in $t$.

The estimate from Thm.~\ref{ThmConvolutionApproximation} is an {\em a posteriori} estimate, in the sense that the error depends on $p_t^\Delta(x, y)$, which itself is the (a priori unknown) solution to a differential equation. One can obtain an {\em a priori} estimate by using the Gaussian estimate from above \cite[Thm.~5.3.4]{hsu}, $p_t^\Delta(x, y) \leq C t^{-n/2+1/2}\e_t(x, y)$, which holds on compat Riemannian manifolds: One gets that one can replace the result of Thm.~\ref{ThmConvolutionApproximation} by the estimate
\begin{equation} \label{APrioriEstimate}
  \left|p_t^L(x, y) - \bigl(\e_{\Delta_1\tau}^\nu * \cdots * \e_{\Delta_N \tau}^\nu\bigr)(x, y) \right| \leq C\,  \e_t(x, y)\,|\tau|^\nu \,t^{3/2+\nu-n/2}.
\end{equation}
This is a weaker statement however, since for example if $(x, y) \in M \bowtie M$, one even has $p_t^\Delta(x, y) \leq C\e_t(x, y)$ for all $0 < t \leq T$, so in this case, the additional factor of $t^{-n/2}$ can be dropped on the right hand side of \eqref{APrioriEstimate} (with the constant being uniform over compact subsets of $M \bowtie M$ in this case).

Similar approximation schemes and their relation to finite-dimensional approximation of path integrals have also been considered by Fine and Sawin, who use these to give a ``path integral proof'' of the Atiyah-Singer index theorem, see \cite{FineSawin}, \cite{FineSawin2} or \cite{FineSawin3}. 

\medskip

In this paper, we use Thm.~\ref{ThmStrongEstimates} to analyze the short time asymptotics of the heat kernel at the cut locus. We show that if the set of minimizing geodesics between $x$ and $y$ is a disjoint union of $k$ submanifolds of the space of finite energy paths connecting $x$ and $y$, having dimensions $d_1, \dots, d_k$ (see Def.~\ref{DefNonDegenerate} below), then under a natural non-degeneracy condition, the heat kernel has an asymptotic expansion of the form
\begin{equation*}
   \frac{p_t^L(x, y)}{\e_t(x, y)} ~\sim~ \sum_{l=1}^k (4\pi t)^{-d_l/2} \sum_{j=0}^\infty t^j \frac{\Phi_{j, l}(x, y)}{j!},
\end{equation*}
as $t\rightarrow 0$. In order to derive this result, we show that the convolution product $\e_{\Delta_1\tau}^\nu * \cdots * \e_{\Delta_N \tau}^\nu$ can be written as an integral over a certain space of piecewise geodesics paths, which can then be evaluated with Laplace's method. See e.g.\ \cite{molchanov}, \cite{NeelStrook} or \cite{InahamaTaniguchi}, wo obtain similar results using methods from stochastic analysis.

\medskip

This paper is organized as follows. First we summarize some facts about the solution theory of the wave equation and introduce the transformation formula, which relates it to the heat equation. Here we also highlight some conditions for the Laplace type operator that suffice to have the transmutation formula valid on complete manifolds and we use the formula to prove some results on essential self-adjointness. Subsequently, in Section~\ref{HeatKernelAsymptotics}, we introduce the Hadamard expansion of the solution operator to the heat equation and combine it with the transmutation formula to prove Thm.~\ref{ThmStrongEstimates}. We also briefly demonstrate how the well-known Gaussian estimates from above and below are derived using this technique. In the next section, we give a proof of Thm.~\ref{ThmConvolutionApproximation}. In a final section, we reformulate this convolution product as a path integral, which is then analyzed to obtain an asymptotic expansion of the heat kernel $p_t^L(x, y)$ also in the case that $x$ and $y$ lie in each other's cut locus. In an appendix, we prove a general version of Laplace's method, which is needed in our considerations.

\medskip

\textbf{Acknowledgements}. I would like to thank Christian Bär, Rafe Mazzeo, Franziska Beitz, Florian Hanisch and Ahmad Afuni for helpful discussions. Furthermore, I am indepted to Potsdam Graduate School, The Fulbright Program, SFB 647 and the Max-Planck-Institute for Mathematics in Bonn for financial support.

\section{The Wave Equation and the Transmutation formula} \label{SectionWaveTransmutation}

Let $M$ be a complete Riemannian manifold of dimension $n$ and let $\V$ be a metric vector bundle over $M$. A Laplace type operator $L$ on $\V$ is a second order differential operator acting on sections of $\V$, which in local coordinates is given by
\begin{equation*}
  L = - \id_{\V} \,g^{ij} \frac{\partial^2}{\partial x^i \partial x^j} + \text{lower order terms},
\end{equation*}
where $(g^{ij})$ is the inverse matrix of the matrix $(g_{ij})$ describing the metric in the local coordinates. Considered as an unbounded operator on $L^2(M, \V)$, a natural domain for $L$ is the space $\mathscr{D}(M, \V) := C^\infty_c(M, \V)$, the space of smooth, compactly supported sections of the bundle $\V$ (which, when necessary, is endowed with the usual test function topology). We say that $L$ is formally self-adjoint if it is symmetric on this domain. 

Given such a Laplace type operator $L$, one can consider the {\em wave equation}
\begin{equation} \label{WaveEquation}
  (\partial_{tt} +L)u_t = 0.
\end{equation}
A fundamental feature of the wave equation is the {\em energy estimate}, which states that for any compact set $K \subseteq M$, any $m \in \R$ and any $T>0$, there exists a constant $\alpha \in \R$ such that for all smooth solutions $u$ of the wave equation with $\mathrm{supp}\, u_0 \subseteq K$, one has
\begin{equation} \label{EnergyEstimate}
  \|u_t\|_{H^m}^2 + \|u_t^\prime\|_{H^{m-1}}^2 \leq e^{\alpha (t-s)} \bigl(\|u_s\|_{H^m}^2 + \|u_s^\prime\|_{H^{m-1}}^2\bigr)
\end{equation}
whenever $-T \leq s \leq t \leq T$ (see e.g.\ \cite[Thm.~8]{BaerTagneWafo}). 

From the theory of wave equations follows that there is a family of solution operators $G_t: \mathscr{D}(M, \V) \longrightarrow \mathscr{D}(M, \V)$ such that for $\psi \in \mathscr{D}(M, \V)$, $u_t := G_t \psi$  solves the wave equation \eqref{WaveEquation} with initial conditions $u_0 = 0$, $u_0^\prime = \psi$. We also have its derivative $G_s^\prime$, which has the property that $u_t := G_t^\prime \psi$ solves the wave equation with initial condition $u_0 = \psi$, $u^\prime_0 = 0$ (see e.g.\ Corollary~14 in \cite{BaerTagneWafo}). 

Instead of the wave equation, we can also consider the {\em heat equation}
\begin{equation} \label{HeatEquation}
  (\partial_{t} +L)u_t = 0.
\end{equation}
Here we only need to specify an initial condition $\psi$ at time zero to have a unique (bounded) solution. This leads to a solution operator $e^{-tL}$, mapping the initial condition $\psi$ to the solution $u_t$. The heat equation is related to the  wave equation  as follows.
  
\begin{theorem}[Transmutation Formula] \label{ThmTransmutation}
   Let $M$ be a complete Riemannian manifold and let $L$ be a Laplace type operator, acting on sections of a metric vector bundle $\V$ over $M$. Suppose that the wave operators $G_t$ and $G_t^\prime$ defined on $\mathscr{D}(M, \V)$ extend to strongly continuous families of operators on $L^2(M, \V)$ satisfying the norm bound
   \begin{equation} \label{NormBoundG}
     \|G_t\|, \|G_t^\prime \| \leq C e ^{\alpha |t|}
   \end{equation}
   for some $C>0$, $\alpha \in \R$. Then setting
      \begin{equation} \label{TransmutationFormulaOp}
     e^{-tL}u = \int_{-\infty}^\infty \gamma_t(s) G^\prime_s u \,\dd s, ~~~~~~\text{with}~~~~~~\gamma_t(s) := (4\pi t)^{-1/2} e^{-s^2/4t}
   \end{equation}
   for $u \in L^2(M, \V)$ defines a strongly continuous semigroup of operators, the infinitesimal generator of which is an extension of $L$ with  $\mathrm{dom}(L)=\mathscr{D}(M, \V)$.
 \end{theorem}

\begin{remark}
Of course, the continuous extensions of $G_t$ respectively $G_t^\prime$, if they exist, are unique, since $\mathscr{D}(M, \V)$ is dense in $L^2(M, \V)$.
\end{remark}

\begin{remark}
The same result is true when $L^2(M, \V)$ is replaced by any Banach space $E$ of distributions containing $\mathscr{D}(M, \V)$ as a dense subset and such that the inclusion of $E$ into $\mathscr{D}^\prime (M, \V)$ is continuous.
\end{remark}

\begin{proof}
Define for $u \in L^2(M, \V)$
\begin{equation*}
  P_t u := \int_{-\infty}^\infty \gamma_t(s) G^\prime_s u \dd s.
\end{equation*}
  By the norm bound on $G_t^\prime$, the integral on the right hand \eqref{TransmutationFormulaOp} converges absolutely for each $t>0$, and $P_t$ is a locally uniformly bounded family of operators. We now verify that $P_t$ is a strongly continuous semigroup. First, because $\gamma_t$ integrates to one over the line, we have
  \begin{equation*}
    \|P_t u - u\|_{L^2} = \left\|\int_{-\infty}^\infty \gamma_t(s) (G^\prime_s u - u)\dd s\right\|_{L^2} \leq \int_{-\infty}^\infty \gamma_t(s) \|G^\prime_s u - u\|_{L^2}.
  \end{equation*}
  for all $u \in L^2(M, \V)$.
  Because $G_s^\prime$ is strongly continuous by assumption and $G^\prime_0 u = u$, the function $\|G^\prime_s u - u\|_{L^2}$ is continuous in $s$ and vanishes at zero. Now $ \|P_t u - u\|_{L^2}\rightarrow 0$ follows from the well-known fact that $\gamma_t \rightarrow \delta_0$ as $t\rightarrow 0$.
  
To verify the semigroup property, we use that for any $s, t \in \R$ and $\psi \in \mathscr{D}(M, \V)$, we have the ``trigonometric formula''
\begin{equation*}
  G_s^\prime G_t^\prime \psi = G_{s+t}^\prime \psi - G_s G_t L \psi,
\end{equation*}
which can easily be verified by fixing $s$ and noticing that both sides satisfy the wave equation with respect to the variable $t$ and with the same initial conditions. The energy estimate \eqref{EnergyEstimate} implies then that their difference must be zero. Now
\begin{equation*}
\begin{aligned}
  (\varphi, P_sP_t\psi)_{L^2} &= \int_{-\infty}^\infty \int_{-\infty}^\infty \gamma_t(u)\gamma_s(v) (\varphi, G^\prime_uG^\prime_v\psi)_{L^2} \,\dd v \dd u\\
  &= \int_{-\infty}^\infty \int_{-\infty}^\infty \gamma_t(u)\gamma_s(v) \bigl((\varphi, G_{u+v}^\prime\psi)_{L^2} - (\varphi, G_uG_v L\psi)_{L^2}\bigr) \dd v \dd u,
\end{aligned}
\end{equation*}
where the integral over each individual term is absolutely convergent by the bound \eqref{NormBoundG} on $G_u$ and $G_u^\prime$.
Because $G_u$ is an odd function of $u$, the term involving $G_u G_v L\psi$ integrates to zero. Therefore
\begin{equation*}
\begin{aligned}
(\varphi, P_sP_t\psi)_{L^2} &= \int_{-\infty}^\infty \int_{-\infty}^\infty \gamma_t(u)\gamma_s(v) (\varphi, G_{u+v}^\prime\psi)_{L^2} \dd v \dd u \\
&= \int_{-\infty}^\infty \left(\int_{-\infty}^\infty \gamma_t(r-v)\gamma_s(v) \dd v\right) (\varphi, G_{r}^\prime\psi)_{L^2}  \dd r\\
&= \int_{-\infty}^\infty \gamma_{s+t}(r) (\varphi, G_{r}^\prime\psi)_{L^2}  \dd r = (\varphi, P_{t+s}\psi)_{L^2}.
\end{aligned}
\end{equation*}
Hence $P_sP_t = P_{t+s}$ on the dense subset $\mathscr{D}(M, \V) \subset L^2(M, \V)$ and by boundedness also on all of $L^2(M, \V)$. This shows that $P_t$ is a strongly continuous semigroup of operators.

To see what the infinitesimal generator of $P_t$ is, notice that for $\psi \in \mathscr{D}(M, \V)$, the estimates on $G_t$ and $G_t^\prime$ justify the calculation
\begin{equation*}
\begin{aligned}
  P_t^\prime \psi &= \int_{-\infty}^\infty \gamma_t^\prime(s) G_s^\prime u\, \dd s = \int_{-\infty}^\infty \frac{\partial^2}{\partial s^2} \gamma_t (s) G_s^\prime \psi\, \dd s \\
  &= \int_{-\infty}^\infty \frac{\partial}{\partial s} \gamma_t (s) G_s L \psi \,\dd s = -\int_{-\infty}^\infty \gamma_t (s) G_s^\prime L \psi \,\dd s = -LP_t \psi.
  \end{aligned}
\end{equation*}
This shows that the infinitesimal generator $\overline{L}$ (which is always closed for a strongly continuous semigroup) is an extension of the operator $L$ with $\mathrm{dom}(L) = \mathscr{D}(M, \V)$.
\end{proof}

In particular, the result is applicable to the compact setting:

\begin{lemma}
  For any Laplace type operator acting on sections of a metric vector bundle over a compact Riemannian manifold, the assumptions of Thm.~\ref{ThmTransmutation} are satisfied.
\end{lemma}

\begin{proof}
  The bound \eqref{NormBoundG} follows directly from the energy estimate \eqref{EnergyEstimate} in this case, since one can take $K = M$ and it is also clear that one can take the same $\alpha$ for each $T$. 
\end{proof}

Furthermore, it is well-known that on a compact manifold, any Laplace type operator $L$ has a unique closed extension that is the generator of a strongly continuous semigroup (this follows e.g.\ from Lemma~2.16 in \cite{bgv}).

A consequence of Thm.~\ref{ThmTransmutation} is the following.

\begin{theorem} \label{ThmSelfadjointCase}
Let $L$ be a formally self-adjoint Laplace type operator, acting on sections of a metric vector bundle over a complete Riemannian manifold. Considered as an unbounded symmetric operator with domain $\mathscr{D}(M, \V)$, $L$ admits at most one self-adjoint extension $\overline{L}$ that generates a strongly continuous semigroup of operators. If there is such an extension, then the assumptions of Thm.~\ref{ThmTransmutation} are satisfied.
\end{theorem}

\begin{proof}

Let $\overline{L}$ be a self-adjoint extension of $L$ that generates a strongly continuous semigroup of operators. By the Hille-Yosida theorem, there exists $\omega \in \R$ such that the spectrum of $\overline{L}$ is contained in $[\omega, \infty)$, and $e^{-t\overline{L}}$ is given in terms of spectral calculus via the absolutely convergent integral
\begin{equation*}
  (u, e^{-t\overline{L}} v)_{L^2} = \int_{-\infty}^\infty e^{-t\lambda} \dd (u, E_\lambda v)_{L^2},
\end{equation*}
for $u, v \in L^2(M, \V)$, where $E_\lambda$ is the spectral measure associated to $\overline{L}$. Consider the entire function
\begin{equation*}
  g_s(\lambda) := \sum_{k=0}^\infty \frac{s^{2k+1} \lambda^k}{(2k+1)!}
\end{equation*}
For $\lambda > 0$, we have $g_s(\lambda) = \sin(s\sqrt{\lambda})/\sqrt{\lambda}$ and $g_s(-\lambda) = \sinh(s\sqrt{\lambda})/\sqrt{\lambda}$, while $g_s^\prime(\lambda) = \cos(s \sqrt{\lambda})$ and $g_s^\prime(-\lambda) = \cosh(s\sqrt{\lambda})$. Hence we obtain that 
\begin{equation*}
\bigl\| g_s|_{[\omega, \infty)}\bigr\|_\infty \leq e^{s \omega}, ~~~~ \bigl\| g_s^\prime|_{[\omega, \infty)}\bigr\|_\infty \leq e^{s \omega}.
\end{equation*}
By standard properties of the functional calculus, one obtains the estimates $\|g_s(\overline{L})\| \leq e^{s\omega}$, $\|g_s^\prime(\overline{L})\| \leq e^{s\omega}$ on the operator norms. 

We now claim that the wave operator $G_s$ on $\mathscr{D}(M, \V)$ is given by $G_s = g_s(\overline{L})|_{\mathscr{D}(M, \V)}$. To see this, notice that for any $\psi \in \mathscr{D}(M, \V)$, $g_s(\overline{L}) \psi$ satisfies the wave equation \eqref{WaveEquation} with initial conditions $g_0(\overline{L})\psi = 0$, $g_0^\prime(\overline{L}) \psi = \psi$. Hence $u_s := G_s\psi - g_s(\overline{L}) \psi$ satisfies the wave equation with initial conditions $u_0 = 0$, $u_0^\prime = 0$, which implies $u_s \equiv 0$ by the energy estimate \eqref{EnergyEstimate}. The same argument shows that $G_s^\prime = g_s^\prime(\overline{L})|_{\mathscr{D}(M, \V)}$.

By the above, $G_s$ and $G_s^\prime$ satisfy the norm bound \eqref{NormBoundG}. To see that $G_s$ and $G_s^\prime$ are strongly continuous, we argue as follows: By Lebesgue's theorem of dominated convergence, one obtains that for all $u, v \in L^2(M, \V)$, one has $(u, G_s v)_{L^2} \rightarrow (u, G_t v)_{L^2}$ as $s \rightarrow t$, i.e.\ $G_sv \rightarrow G_t v$ weakly. Similarly, $\|G_s v\|_{L^2} = (v, G_s^2 v)_{L^2} \longrightarrow (v, G_t^2 v)_{L^2} = \|G_t v\|_{L^2}$. Now it is well-known that in Hilbert spaces, weak convergence plus convergence {\em of} norms implies convergence {\em in} norm, so we obtain $G_s v \rightarrow G_t v$ in $L^2(M, \V)$. This shows that $G_s$ is strongly continuous and the argument for $G_s^\prime$ is the same.

Now by Thm.~\ref{ThmTransmutation}, there is some extension $\overline{L}_2$ of $L$ with domain containing $\mathscr{D}(M, \V)$ that generates a strongly continuous semigroup of operators given by the transmutation formula \eqref{TransmutationFormulaOp}. However, by Fubini's theorem, 
\begin{equation*}
  \int_{-\infty}^\infty \gamma_t(s) (u, G^\prime_s v)_{L^2} \dd s = \int_{-\infty}^\infty  \left(\int_{-\infty}^\infty \gamma_t(s) g_s(\lambda)  \dd s \right) \dd (u, E_\lambda v)_{L^2},
\end{equation*}
where the interior integral is easily found to equal $e^{-t\lambda}$, e.g.\ by expanding $g_s(\lambda)$ into its power series and using standard formulas for the moments of a one-dimensional Gaussian measure (see e.g.\ Lemma~2.12 in \cite{bgv}). Therefore, the semigroup generated by $\overline{L}_2$ equals the semigroup generated by $\overline{L}$.

These arguments show that self-adjoint extension of $L$ generating a strongly continuous semigroup of operators, this semigroup is given by the transmutation formula \eqref{TransmutationFormula}. However, this formula does not depend on the self-adjoint extension (because the operator $G_s$ doesn't), so any two strongly continuous operator semigroups generated by self-adjoint extensions of $L$ must concide. But this implies that also the self-adjoint extensions coincide, because the infinitesimal generator of an operator semigroup is unique.
\end{proof}

\begin{example}
 For example, if $L = \nabla^*\nabla + V$ for some connection $\nabla$ on $\V$ and a symmetric endomorphism field $V$ that is bounded from below (meaning that there exists $\omega \in \R$ such that $\<w, Vw\> > \omega$ for all $w \in \V$), then $L$ has a self-adjoint extension that generates a strongly continuous semigroup. Namely, because for $u \in L^2(M, \V)$, 
  \begin{equation*}
    (u, Lu)_{L^2} = \|\nabla u\|_{L^2} + (u, Vu)_{L^2} \geq \omega \|u\|_{L^2},
  \end{equation*}
  the operator $L$ is semi-bounded and it is well-known that it has a self-adjoint extension, called Friedrich extension (see e.g.\ \cite[VII.2.11]{werner95}), which satisfies the same bound and therefore generates an operator semigroup by functional calculus. We obtain that in this setting, the Friedrichs extension is the only self-adjoint extension that is the generator of a strongly continuous semigroup.
  
  In particular, this applies to $\Delta = d^* d$, the Laplace-Beltrami operator acting on functions. 
 \end{example}
 
 \begin{example}
 The Hodge Laplacian $L = (d + d^*)^2$ on differential forms is a positive operator and hence has a self-adjoint extension generating a strongly continuous semigroup by the same argument. By Thm.~\ref{ThmSelfadjointCase}, this is the only self-adjoint extension generating a strongly continuous semigroup of operators. In fact, it is the only self-adjoint extension, by Thm.~2.4 in \cite{Strichartz}. For the same reason, for any self-adjoint Dirac-type operator $D$, the corresponding Laplacian $D^2$ has a unique self-adjoint extension generating a strongly continuous semigroup of operators. Also in this case, it is known that $D$ and $D^2$ are even essentially self-adjoint (i.e.\ they have unique self-adjoint extensions), compare \cite{WolfDirac}.
\end{example}

\begin{example}
In contrast, there are formally self-adjoint Laplace type operators that do not admit any self-adjoint extension. For example, the operator $L = -\Delta - x^4$ on $M = \R$ does not admit a self-adjoint extension (see Ex.~3 on p.~86 in \cite{SchroedingerEquation}). There are also essentially self-adjoint Laplace type operators which do not generate a strongly continuous family of operators, see e.g.\ \cite{SimonFeynmanKac}.
\end{example}

\begin{remark}
Our observations show that matters can be quite subtle on general complete manifolds: A formally self-adjoint Laplace type operator need not have a self-adjoint extension, nor need it be unique. Furthermore, not all self-adjoint extensions generate a strongly continuous semigroup of operators (they do if and only if the spectrum is bounded from below). However, there is at most one self-adjoint extension that generates a strongly continuous semigroup of operators. We do not know of an example of a formally self-adjoint Laplace type operator that admits two different self-adjoint extensions, one of which generates a strongly continuous semigroup and the other doesn't (by Thm.~\ref{ThmSelfadjointCase}, not both of them can generate a strongly continuous semigroup of operators).
\end{remark}

\section{Heat Kernel Asymptotics} \label{HeatKernelAsymptotics}

In this section, we prove the following more general version of Thm.~\ref{ThmStrongEstimates}.

\begin{theorem}[Strong Heat Kernel Asymptotics] \label{ThmStrongEstimates2}
Let $L$ be a Laplace type operator, acting on sections of a vector bundle $\V$ over a complete Riemannian manifold $M$. Suppose that the assumptions of Thm.~\ref{ThmTransmutation} are satisfied (e.g.\ when $M$ is compact or $L$ is formally self-adjoint and semi-bounded). Then for any compact subset $K$ of $M \bowtie M$, any $T>0$ and any numbers $\nu, k, l, m \in \N_0$, there exists a constant $C>0$ such that 
\begin{equation*}
  \left| \frac{\partial^k}{\partial t^k}\nabla_x^l \nabla_y^m\left\{\frac{p_t^L(x, y)}{\e_t(x, y)} - \sum_{j=0}^\nu t^j \frac{\Phi_j(x, y)}{j!}\right\}\right| \leq C t^{\nu + 1-k}
\end{equation*}
for all $(x, y) \in K$, whenever $0 < t \leq T$. Here $\Phi_j(x, y)$ are certain smooth sections of the bundle $\V\boxtimes \V^*$ over $M \bowtie M$.
\end{theorem}

The proof will use the transmutation formula \eqref{TransmutationFormula}, which in terms of integral kernels translates into 
\begin{equation} \label{TransmutationFormula}
  p_t^L(x, y) :=  \int_{-\infty}^\infty \gamma_t(s)\,G^\prime(s, x, y) \,\dd s,
\end{equation}
where $G^\prime(s, x, y)$ is the Schwartz kernel of the operator $G_s$. This integral is meant as a distributional integral, i.e.\ for test functions $\varphi \in \mathscr{D}(M, \V^*)$, $\psi \in \mathscr{D}(M, \V)$, we set
\begin{equation}\label{TransmutationFormula2}
  p_t^L [ \varphi \otimes \psi] := \int_{-\infty}^\infty \gamma_t(s)\,G_s^\prime[\varphi \otimes \psi] \,\dd s
\end{equation}
Let us verify that this indeed defines a distribution on $M \times M$ for each $t>0$, provided that the assumptions of Thm.~\ref{ThmTransmutation} hold. Namely, by the estimate on $G_s^\prime$, we have
\begin{equation} \label{EstimateIC}
  \bigl|G_s^\prime[ \varphi \otimes \psi]\bigr| = \bigl| (\varphi, G_s \psi)_{L^2}\bigr| \leq C e^{\alpha|s|} \|\varphi\|_{L^2}\|\psi\|_{L^2},
\end{equation}
which shows that the integral \eqref{TransmutationFormula} is absolutely convergent. We furthermore have
\begin{equation*}
  \bigl| p_t^L[\varphi \otimes \psi] \bigr| \leq \left(C \int_{-\infty}^\infty \gamma_t(s) e^{\alpha|s|}\dd s\right) \|\varphi\|_{L^2} \|\psi\|_{L^2}
\end{equation*}
for all $\varphi\in \mathscr{D}(M, \V^*)$, $\psi \in \mathscr{D}(M, \V)$, which shows that $p_t^L$ is indeed a well-defined distribution.

\medskip

The wave kernel $G(t, x, y)$ has an asymptotic expansion, the Hadamard expansion, which describes its singularity structure. To state the result, we introduce the {\em Riesz distributions} $R(\alpha; t, x, y) \in \mathscr{D}(M \bowtie M)$. Namely, for $\mathrm{Re}(\alpha)> n+1$, we set
\begin{equation*}
  R(\alpha;t, x, y) := C(\alpha)\, \mathrm{sign}(t) \,\bigl(t^2 - d(x, y)^2\bigr)_+^{\frac{\alpha-n-1}{2}}, ~~~~~ C(\alpha) := \frac{2^{1-\alpha} \pi^{\frac{1-n}{2}}}{\Gamma\left(\frac{\alpha}{2}\right)\Gamma\left(\frac{\alpha-n+1}{2}\right)},
\end{equation*}
where $(t^2 - d(x, y)^2)_+$ denotes the positive part. Hence $R(\alpha;t, x, y)$ is zero whenever $|t| \leq d(x, y)$ (the constant $C(\alpha)$ here equals the constant $C(\alpha, n+1)$ in Def.~1.2.1 of \cite{bgp} because our spacetime $\R \times M$ is $n+1$-dimensional. The distributions $R(\alpha)$ discussed here are related to the distributions $R_\pm(\alpha)$ in Section~1.4 of \cite{bgp} by $R(\alpha) = R_+(\alpha) - R_-(\alpha)$). For $\mathrm{Re}(\alpha) > n+1$, the $R(\alpha; t, x, y)$ are then continuous functions on $\R \times M \bowtie M$ and one can show that they define a holomorphic family of distributions on $\{\mathrm{Re}(\alpha) > n+1\}$ that has a holomorphic extension to all of $\C$  \cite[Lemma~1.2.2~(4)]{bgp}. This defines $R(\alpha; t, x, y) \in \mathscr{D}^\prime(\R \times M \bowtie M)$ for all $\alpha \in \C$.

Now on $M \bowtie M$, the distribution $G(t, x, y)$ has the asymptotic expansion \cite[Ch.~2]{bgp}
\begin{equation} \label{WaveAsymptotics}
  G(t, x, y) ~\sim~ \sum_{j=0}^\infty \Phi_j(x, y) R(2+2j; t, x, y),
\end{equation}
where the $\Phi_j(x, y) \in C^\infty(M \bowtie M, \V \boxtimes \V^*)$ are coeffients determined by  certain transport equations. The asymptotic expansion \eqref{WaveAsymptotics} is meant in the sense that the difference
\begin{equation} \label{DefinitionDifferenceTerm}
  \delta^\nu(t, x, y) := G(t, x, y) - \sum_{j=0}^\nu \Phi_j(x, y) R(2+2j; t, x, y)
\end{equation}
can be made arbitrarily smooth by increasing the number $\nu$ of correction terms; in fact, $\delta^\nu \in C^k(\R \times M \bowtie M, \V \boxtimes \V^*)$ whenever $\nu \geq (n+1)/2 + k$ \cite[Prop.\ 2.5.1]{bgp}. Furthermore, the fact that the wave equation has finite propagation speed (i.e.\ $G(t, x, y) \equiv 0$ on the region where $|t| < d(x, y)$) implies that when $\nu$ is so large that $\delta^\nu$ is $C^k$, one has the estimate 
\begin{equation} \label{ErrorEstimateWaveApprox}
  \left|\frac{\partial^j}{\partial t^j}\nabla^l_x \nabla^m_y \delta^\nu(t, x, y) \right| \leq C \bigl( t^2-d(x, y)^2\bigr)_+^{(k-j-l-m)/2}
\end{equation}
uniformly over compact subsets of $M \bowtie M$ and $t \leq T$, whenever $k\geq l+m$ (compare \cite[Thm.\ 2.5.2]{bgp}).

\begin{lemma} \label{LemmaRieszIntegral}
For all $j \in \N_0$, $t>0$ and all $(x, y) \in M \bowtie M$, we have
\begin{equation}
  \frac{1}{2t} \int_{-\infty}^\infty\gamma_t(s) R(2+2j; s, x, y) \,s\,\dd s 
 = \e_t(x, y) \frac{t^j}{j!},
\end{equation}
where $\e_t(x, y)$ is the Euclidean heat kernel, defined in \eqref{EuclideanHeatKernel}. In particular, the distributional integral on the left hand side actually yields a smooth function.
\end{lemma}

\begin{proof}
For $\mathrm{Re}(\alpha)>n+1$, consider the absolutely convergent integral
\begin{equation*}
\begin{aligned}
  \frac{1}{2t} \int_{-\infty}^\infty \gamma_t(s) R(\alpha; s, x, y) \,s\,\dd s 
  &= \frac{C(\alpha)}{2t}\int_{-\infty}^\infty \gamma_t(s) \bigl(s^2 - d(x, y)^2\bigr)_+^{\frac{\alpha-n-1}{2}} |s| \,\dd s\\
  &= \frac{C(\alpha)}{t}\int_{0}^\infty \gamma_t(s) \bigl(s^2 - d(x, y)^2\bigr)_+^{\frac{\alpha-n-1}{2}} s \,\dd s\\
  &= \frac{C(\alpha)}{t}\int_{d(x, y)}^\infty \gamma_t(s) \bigl(s^2 - d(x, y)^2\bigr)^{\frac{\alpha-n-1}{2}} s \,\dd s
\end{aligned}
\end{equation*}
Performing the substitution $u^2 = s^2-d(x, y)^2$ which transforms the interval $(d(x, y), \infty)$ into the interval $(0, \infty)$, we have $s \dd s = u \dd u$. Therefore, we obtain
\begin{align*}
  \int_{d(x, y)}^\infty \gamma_t(s) \bigl(s^2 - d(x, y)^2\bigr)^{\frac{\alpha-n-1}{2}} s \,\dd s 
  &= \gamma_t\bigl(d(x, y)\bigr) \int_{0}^\infty e^{-u^2/4t} u^{\alpha-n} \,\dd u.
\end{align*}
Now, substituting $u^2/4t = r$, the integral can be brought into the form of a gamma-integral, giving
\begin{equation*}
  \int_{0}^\infty e^{-u^2/4t} u^{\alpha-n} \,\dd u = t^{1/2}(4t)^{\frac{\alpha-n}{2}} \int_0^\infty e^{-r} r^{\frac{\alpha-n-1}{2}} \dd r =  t^{1/2}(4t)^{\frac{\alpha-n}{2}} \Gamma\left( \frac{\alpha-n+1}{2}\right).
\end{equation*}
Put together, we arrive at
\begin{equation}\begin{aligned} \label{RieszDistributionIntegral}
 \frac{1}{2t} \int_{-\infty}^\infty\gamma_t(s) R(\alpha; s, x, y) \,s\,\dd s 
 &= \gamma_t\bigl(d(x, y)\bigr) \frac{C(\alpha)}{t} t^{1/2}(4t)^{\frac{\alpha-n}{2}}\Gamma\left( \frac{\alpha-n+1}{2}\right)\\
 &= \e_t(x, y)\frac{t^{\frac{\alpha-2}{2}}}{\Gamma(\alpha/2)}.
\end{aligned}
\end{equation}
Until now, we have restricted ourselves to the case $\mathrm{Re}\,\alpha>n+1$. However, for both sides of the last equation, if we pair them with a test function $\varphi \in \mathscr{D}(M \bowtie M)$, the result will be an entire holomorphic function in $\alpha$. Because they coincide for $\mathrm{Re}\,\alpha > n+1$, they must coincide everywhere, by the identity theorem for holomorphic functions.

The statement of the lemma is the particular result for $\alpha = 2 + 2j$, $j \in \N_0$.
\end{proof}

\begin{proof}[of Thm.\ \ref{ThmStrongEstimates2}]
Integrating by parts in \eqref{TransmutationFormula}, which is justified by the estimate \eqref{NormBoundG}, we obtain
\begin{equation}\begin{aligned} \label{TransmutationKernels}
  p_t^L(x, y) &= \int_{-\infty}^\infty \gamma_t(s) \,G^\prime(s, x, y) \,\dd s \\
  &= \int_{-\infty}^\infty \gamma_t(s) \, G(s, x, y) \,\frac{s}{2t}\,\dd s
\end{aligned}\end{equation}
where the identity is to be interpreted in the distributional sense. Now for any $\nu \in \N$, we have
\begin{equation*}
  p_t^L(x, y) = \sum_{j=0}^\nu \frac{\Phi_j(x, y)}{2t } \int_{-\infty}^\infty \gamma_t(s) R(2+2j; s, x, y)\,s\, \dd s
  + \frac{1}{2t} \int_{-\infty}^\infty \gamma_t(s) \delta^\nu(s, x, y) \,s\,\dd s,
\end{equation*}
where $\delta^\nu(t, x, y)$ is in $C^k$ whenever $\nu \geq (n+1)/2 + k$. By Lemma~\ref{LemmaRieszIntegral}, the first term evaluates  to
\begin{equation*}
  \sum_{j=0}^\nu \frac{\Phi_j(x, y)}{2t } \int_{-\infty}^\infty \gamma_t(s) R(2+2j; s, x, y)\,s\, \dd s
  = \e_t(x, y) \sum_{j=0}^\nu t^j \frac{\Phi_j(x, y)}{j!}.
\end{equation*}
It remains to estimate the error term. Because $G_t = - G_{-t}$ and the Riesz distributions are odd in $t$, the remainder term $\delta^\nu(t, x, y)$ is an odd function in the $t$ variable. We conclude
\begin{equation*}
  r^\nu(t, x, y) := \frac{1}{2t} \int_{-\infty}^\infty \gamma_t(s) \delta^\nu(s, x, y)\,s\, \dd s = \frac{1}{t} \int_{d(x, y)}^\infty \gamma_t(s) \delta^\nu(s, x, y) \, s\,\dd s,
\end{equation*}
as $\delta^\nu(s, x, y) = 0$ if $s < d(x, y)$, because of \eqref{ErrorEstimateWaveApprox}.
Substituting $s = \sqrt{u^2 + d(x, y)^2}$ as before, one obtains
\begin{equation*}
  r^\nu(t, x, y) = \frac{\gamma_t\bigl(d(x, y)\bigr)}{t} \int_{0}^\infty e^{-\frac{u^2}{4t}} \delta^\nu(\sqrt{u^2 + d(x, y)^2}, x, y) \,u\,\dd u.
\end{equation*}
Setting $\tilde{\delta}^\nu(u, x, y) := \delta_\nu(\sqrt{u^2 + d(x, y)^2}, x, y)$ one has that $\tilde{\delta}^\nu$ is $C^k$ whenever $\delta^\nu$ is $C^k$, and from \eqref{ErrorEstimateWaveApprox} follows the estimate
\begin{equation} \label{RemainderEstimateSubstituted}
  \Bigl|\frac{\partial^i}{\partial u^i} \nabla^l_x \nabla^m_y \tilde{\delta}^\nu(u, x, y) \Bigr| \leq C u^{k-i-l-m}.
\end{equation}
which is valid whenever $k\geq i+l+m$ and uniform over $x, y$ in compact subsets of $M \bowtie M$ and $u\leq T$. Now the function $e^{-u^2/4t}$ satisfies
\begin{equation*}
  -\frac{2t}{u}\frac{\partial}{\partial u} e^{-u^2/4t} = e^{-u^2/4t},
\end{equation*}
hence for any $r, l, m \in \N_0$, one obtains
\begin{equation*}
  \nabla^l_x \nabla^m_y\Bigl\{\frac{r^\nu(t, x, y)}{\e_t(x, y)}\Bigr\} = \frac{(2t)^r}{t } \int_{0}^\infty \gamma_t(u) \frac{\partial}{\partial u} \Bigl(\frac{1}{u} \frac{\partial}{\partial u} \Bigr)^{r-1} \nabla^l_x \nabla^m_y\tilde{\delta}^\nu(u, x, y) \dd u.
\end{equation*}
if $\nu$ is large enough, depending on $l$ and $m$.
The estimate \eqref{RemainderEstimateSubstituted} shows that these manipulations make sense when $\nu$ is large enough, i.e.\ in this case, the integral is absolutely convergent and uniformly bounded independent of $t$. Therefore, for any $\nu$, one can find $\tilde{\nu} \geq \nu$ large enough so that
\begin{equation*}
  \left|\frac{\partial^i}{\partial t^i}\nabla^l_x \nabla^m_y\left\{\frac{p_t^L(x, y)}{\e_t(x, y)} - \sum_{j=0}^{\tilde{\nu}} t^j \frac{\Phi_j(x, y)}{j!}\right\}\right|
  \leq C t^{\nu+1 - i},
\end{equation*}
where the estimate is uniform for $(x, y)$ in compact subsets of $M \bowtie M$ and $t \leq T$. However, the calculation
\begin{align*}
  &\left|\frac{\partial^i}{\partial t^i}\nabla^l_x \nabla^m_y\left\{\frac{p_t(x, y)}{\e_t(x, y)} - \sum_{j=0}^{\nu} t^j \frac{\Phi_j(x, y)}{j!}\right\}\right|\\
  &~~~~\leq \left|\frac{\partial^i}{\partial t^i}\nabla^l_x \nabla^m_y\left\{\frac{p_t(x, y)}{\e_t(x, y)} - \sum_{j=0}^{\tilde{\nu}} t^j \frac{\Phi_j(x, y)}{j!}\right\}\right|+ \left|\frac{\partial^i}{\partial t^i}\sum_{j={\nu}+i}^\nu t^j \frac{\nabla^l_x \nabla^m_y\Phi_j(x, y)}{j!}\right| \leq C^\prime t^{\nu+1-i}
\end{align*}
shows that in fact $\tilde{\nu}= \nu$ suffices.
\end{proof}

\begin{corollary} \label{CorollarySymmetry}
  If the Laplace-type operator $L$ is formally self-adjoint, then the heat kernel coefficients satisfy the symmetry relation
  \begin{equation*}
    \Phi_j(x, y) = \bigl(\Phi_j^*(y, x)\bigr)^*,
  \end{equation*}
  where $\Phi_j^*$ are the heat kernel coefficients for the heat kernel $p_t^*$ of ${L^*}$, the formally adjoint operator and $(\Phi_t^*(y, x))^*$ denotes the fiberwise metric adjoint of $\Phi_t^*(y, x)$. 
\end{corollary}

\begin{proof}
By Thm.\ \ref{ThmStrongEstimates}, this follows from the fact that the heat kernel itself satisfies the same symmetry relation by Prop.~2.17 (2) in \cite{bgv}. Note that this argument does not work if one only knows \eqref{WeakAsymptotics}.
\end{proof}

\begin{remark}
Corollary~\ref{CorollarySymmetry} is not at all obvious from the defining transport equations for the $\Phi_j$.
  The result was previously proved in the scalar case by Moretti \cite{moretti}, \cite{moretti2} for the heat equation and the Hadamard coefficients by approximating the given metric by real analytic metrics. However, for the heat kernel coefficients, this comes out directly from Thm.\ \ref{ThmStrongEstimates}.
\end{remark}

In the remainder of this section, we demonstrate how to obtain Gaussian estmimates on $p_t^L$ using our techniques.

\begin{theorem}[Gaussian upper Bound] \label{ThmGaussianUpperBounds}
Let $L$ be a formally self-adjoint Laplace type operator acting on sections of a vector bundle $\V$ over a compact manifold $M$ and let $p_t$ be its heat kernel. Then for any $T>0$, there exists a constant $C>0$ such that
  \begin{equation*}
     \Bigl|\frac{\partial^j}{\partial t^j} \nabla_x^m \nabla_y^l p_t(x, y)\Bigr| \leq C t^{-(n+2j+m+l+1)} e^{-\frac{d(x, y)^2}{4t}}
\end{equation*}
for all $x, y \in M$, whenever $t \leq T$.  
\end{theorem}

\begin{remark}
In fact, the upper bound can be improved to have a pre-factor of $t^{-n+1/2}$ instead of $t^{-n-1}$ in the case $j=m=l=0$ \cite[Thm.~5.3.4]{hsu}. This result is then sharp, as seen e.g.\ by the example of two antipodal points of a sphere \cite[Example~5.3.3]{hsu}. Of course, if $(x, y) \in M \bowtie M$, then the correct exponent is $t^{-n/2}$ near $(x, y)$, by Thm.\ \ref{ThmStrongEstimates}.
\end{remark}

\begin{proof}[sketch]
We use the transmutation formula \eqref{TransmutationFormula}. Since $G^\prime_t$ is a commuting and uniformly bounded family of self-adjoint linear operators on $L^2(M, \V)$, its Schwartz kernel $G^\prime(t, x, y)$ is a distribution of order at most $n+1$ on $M\times M \times \R$. The wavefront set of $G^\prime(t, x, y)$ is contained in the characteristics of the wave operator $\partial_{tt} + L$, (i.e.\ the ``light cone'') which are transversal to the submanifolds $\R \times \{(x, y)\} \subset \R \times M \times M$. Therefore, one can restrict $G^\prime$ to these submanifolds, so that for $(x, y) \in M \times M$ fixed, $G^\prime(t, x, y)$ is a distribution on $\R$ of order at most $n+1$ in the variable $t$. Similarly, $\frac{\partial^j}{\partial t^j}\nabla_x^k \nabla_y^l G^\prime(t, x, y)$ is a distribution of order at most $k:=n+1+m+l+j$ on $\R$. This means that
\begin{equation}
  \frac{\partial^j}{\partial t^j}\nabla_x^m \nabla_y^l G^\prime(t, x, y) = \frac{\partial^{k}}{\partial t^k}f(t, x, y)
\end{equation}
in the sense of distributions for some $L^1$ function $f(s, x, y)$. Integration by parts gives
\begin{equation*}
  \nabla_x^m\nabla_y^l p_t(x, y) = (-1)^k\int_{-\infty}^\infty \frac{\partial^k}{\partial s^k} \gamma_t(s) \,f(s, x, y) \dd s,
\end{equation*}
which gives a pre-factor of order $-k$ in $t$. Here, the integration by parts is justified by standard energy estimates. Differentiating $j$ times by $t$ gives another pre-factor of order $-2j$ in $t$.

The result now follows from the fact that $G^\prime(s, x, y)$ and hence also $f(s, x, y)$ is equal to zero for $|s| < d(x, y)$, by finite propagation speed of the wave equation.
\end{proof}

\begin{theorem}[Gaussian lower Bound] \label{ThmGaussianLowerBounds}
Let $M$ be a compact Riemannian manifold and let $L$ be a scalar Laplace type operator $L$, i.e.\ a Laplace type operator acting on functions on $M$. Then for any $T>0$, there exists a constant $C>0$ such that
  \begin{equation*}
    \e_t(x, y) \leq C p_t^L(x, y)
\end{equation*}
 for all $x, y \in M$, whenever $t \leq T$.  
\end{theorem}

\begin{proof}[sketch]
For some $\nu \geq 1$, let $\e_t^\nu(x, y)$ be the approximate heat kernel of $L$ defined in \eqref{ApproximateHeatKernel}.
For $N$ large and some $\nu \geq 1$, set $\tau := \{ 0 < 1/N < 2/N < \dots < (N-1)/N < 1 \}$ for the equidistant partition of the interval $[0, 1]$ with $N+1$ nodes. Because of Thm.~\ref{ThmStrongEstimates}, we have
\begin{equation*}
  \bigl(\underbrace{\e_{t/N}^\nu* \cdots * e_{t/N}^\nu}_{N~\text{times}}\bigr)(x, y) \leq \bigl(C p_{t/N}^L * \cdots * Cp_{t/N}\bigr)(x, y) = C^N p_t^L(x, y)
\end{equation*}
for some $C>0$.
On the other hand, by Lemma~\ref{LemmaPathIntegral} below, the convolution product $\e_{t/N}^\nu* \cdots * e_{t/N}^\nu$ can be written as an integral over the manifold $H_{xy;\tau}(M)$ of piecewise geodesics (introduced in Section~\ref{SectionAsymptoticsFar}),
\begin{equation*}
   \bigl(\e_{t/N}* \cdots * e_{t/N}^\nu\bigr)(x, y) = (4 \pi t)^{-nN/2} \int_{H_{xy;\tau}(M)} e^{-E(\gamma)/2t} \Upsilon_{\tau, \nu}(t, \gamma) \dd \gamma,
\end{equation*}
where $E$ is the energy functional \eqref{EnergyFunctional} and $\Upsilon_{\tau, \nu}(t, \gamma)$ is some smooth function, depending polynomially on $t$. An investigation of the integral using Laplace's method (see Appendix~\ref{AppendixLaplace}) shows that 
\begin{equation*}
  \frac{ \bigl(\e_{t/N}* \cdots * e_{t/N}^\nu\bigr)(x, y)}{\e_t(x, y)} \geq \varepsilon,
\end{equation*}
where $\varepsilon>0$ is independent of $x$ and $y$, where one uses that $\Upsilon_{\tau, \nu}(0, \gamma)>0$ for all minimal geodesics $\gamma$ connecting $x$ and $y$, if $N$ is large enough.
\end{proof}

\begin{remark}
There is a rich literature containing Gaussian bounds for the Laplace-Beltrami operator. In the stochastic literature, two-sided estimates can be found e.g.\ in \cite{molchanov}, \cite[Thm.~5.3.4]{hsu} and \cite{BarilariBoscainNeel}. Using analytic methods, the Gaussian estimate from above is derived e.g.\ in \cite{daviespang}, \cite[Thm.~15.14]{grigoryan} and \cite{coulhonsikora}. 
\end{remark}

\section{Convolution Approximation}

In this section, we prove Thm.~\ref{ThmConvolutionApproximation}. Throughout, $M$ is a compact Riemannian manifold and $L$ is a Laplace type operator, acting on sections of a metric vector bundle $\V$ over $M$.

The proof  relies the following lemma.

\begin{lemma} \label{LemmaLargeOscillations}
For any $0 < \varepsilon < 1$ and all $R, T>0$, there exist constants $C, \delta >0$ such that for all $x, y \in M$, we have
\begin{equation*}
 \int_{d(z_0, z_1) \geq R} p_{s_0}^\Delta(x, z_0) p_{s_1-s_0}^\Delta(z_0, z_1) p_{t-s_1}^\Delta(z_1, y)\, \dd(z_0, z_1) < C e^{-(1-\varepsilon)\frac{R^2}{4(s_1-s_0)}}p_t^\Delta(x, y)
\end{equation*}
whenever $0 \leq s_0 < s_1 \leq t\leq T$ and $s_1-s_0 \leq t \delta$. Here $p_t^\Delta$ denotes the heat kernel of the Laplace-Beltrami operator on $M$.
\end{lemma}

\begin{proof} 
Set
\begin{equation*}
  I := \frac{1}{p_t^\Delta(x, y)} \int_{d(z_0, z_1) \geq R} p_{s_0}^\Delta(x, z_0) p_{s_1-s_0}^\Delta(z_0, z_1) p_{t-s_1}^\Delta(z_1, y)\, \dd(z_0, z_1)
\end{equation*}
and put
\begin{equation*}
  \varphi(r) = \begin{cases} 0 & r < R\\ 1 &r \geq R. \end{cases}
\end{equation*}
By  Thm.~\ref{ThmGaussianUpperBounds} and Thm.~\ref{ThmGaussianLowerBounds}, there exist constants $C_1, C_2>0$ such that for all $0<t\leq T$ and all $x, y \in M$, we have
\begin{equation} \label{GaussianBoundsInLemma}
  C_1 t^{-n/2} e^{-\frac{d(x,y)^2}{4t}} \leq p^\Delta_{t}(x, y) \leq C_2 t^{-n-1} e^{-\frac{d(x,y)^2}{4t}}.
\end{equation}
Using this, we obtain
\begin{align*}
 I \leq \frac{C_2(s_1-s_0)^{-n-1} }{p^\Delta_t(x, y)} \int_M\int_Me^{- \frac{d(z_0, z_1)^2}{4(s_1-s_0)}} p^\Delta_{s_0}(x, z_0)p^\Delta_{t-s_1}(z_1, y) \varphi\bigl(d(z_0, z_1)\bigr) \dd z_0\dd z_1.
\end{align*}
Now set for any $\varepsilon^\prime$ with $0<\varepsilon^\prime < \varepsilon$
\begin{equation} \label{ChoiceOfDelta}
  \delta := \varepsilon^\prime \frac{R^2}{\mathrm{diam}(M)^2}.
\end{equation}
Then on the set where $\varphi(d(z_0, z_1))\neq 0$, i.e.\ $d(z_0, z_1)\geq R$, we have whenever $s_1-s_0 \leq t\delta$ the estimate
  \begin{equation*}
  \begin{aligned}
    \frac{d(z_0, z_1)^2}{4(s_1-s_0)} - \frac{d(x,y)^2}{4t} 
    &\geq \frac{R^2}{4(s_1-s_0)} - \frac{d(x, y)^2\delta}{4(s_1-s_0)}
    = \frac{R^2}{4(s_1-s_0)} - \varepsilon^\prime\frac{R^2d(x, y)^2}{4(s_1-s_0)\mathrm{diam}(M)^2}\\
    &\geq \bigl(1-\varepsilon^\prime\bigr)\frac{R^2}{4(s_1-s_0)}.
  \end{aligned}
  \end{equation*}
  Hence under this restriction on $s_1-s_0$ and using that the function $p_t^\Delta(x, -)$ integrates to one for each $x \in M$, as well as \eqref{GaussianBoundsInLemma}, we have for each $0<t \leq T$ that
  \begin{align*}
 I &\leq \frac{C_2(s_1-s_0)^{-n-1}}{p_t^\Delta(x, y)} e^{ -(1-\varepsilon^\prime)\frac{R^2}{4(s_1-s_0)}-\frac{d(x, y)^2}{4t}} \int_M \int_M p^\Delta_{s_0}(x, z_0)p^\Delta_{t-s_1}(z_1, y) \dd z_0 \dd z_1 \\
 &\leq  C_2(s_1-s_0)^{-n-1} e^{-(1-\varepsilon^\prime)\frac{R^2}{4(s_1-s_0)}} T^{n/2} \frac{t^{-n/2} e^{-\frac{d(x, y)^2}{4t}}}{p_t^\Delta(x, y)} \\
&\leq {C_3} (s_1-s_0)^{-n-1} e^{-(1-\varepsilon^\prime)\frac{R^2}{4(s_1-s_0)}}  < C_4  e^{-(1-\varepsilon)\frac{R^2}{4(s_1-s_0)}},
  \end{align*}
  if the constants $C_3$, $C_4$ are chosen appropriately.
\end{proof}

\begin{remark}
 The proof above shows that one can choose $\delta$ as in Thm.~\ref{ThmConvolutionApproximation} in order that the statement of Lemma~\ref{LemmaLargeOscillations} holds. 
\end{remark}

We can now prove Thm.~\ref{ThmConvolutionApproximation}.

\begin{proof}[of Thm.~\ref{ThmConvolutionApproximation}]
Throughout the proof, write $\Delta_j := \Delta_j \tau$ for abbreviation. By the Markhov property of the heat kernel, we have $p_t^L = p_s^L * p^L_{t-s}$ for all $0 < s < t$. We obtain that
\begin{equation*}
  p_t^L - \e_{\Delta_1}^\nu * \cdots * \e_{\Delta_N}^\nu = \sum_{j=1}^N p_{\tau_{j-1}}^L *\cdots * p_{\Delta_{j-1}}^L * \bigl( p_{\Delta_j}^L - \e_{\Delta_j}^\nu\bigr)* \e_{\Delta_{j+1}}^\nu * \cdots * \e_{\Delta_N}^\nu,
\end{equation*}
since the sum on the right hand side telescopes.
By the Hess-Schrader-Uhlenbrock estimate \cite{HessSchraderUhlenbrock}, we have $|p_t^L| \leq e^{\alpha t} p_t^\Delta$ for some constant $\alpha \in \R$, where $p_t^\Delta$ denotes the heat kernel of the Laplace-Beltrami operator (here we use self-adjointness of the operator $L$). Similarly,
\begin{equation*}
\begin{aligned}
  \bigl|\e_t^\nu(x, y) \bigr| &\leq \bigl|p_t^L(x, y)\bigr| + \bigl| e_t^\nu(x, y) - \chi\bigl(d(x, y)\bigr) p_t^L(x, y)\bigr | \leq e^{\alpha t} p_t^\Delta(x, y) + C t^{\nu + 1} \e_t(x, y)\\
  &\leq e^{\alpha^\prime t} p_t^{\Delta}(x, y)
\end{aligned}
\end{equation*}
for some $\alpha^\prime > 0$, where we used Thm.~\ref{ThmStrongEstimates} and the Gaussian estimate from below, Thm.~\ref{ThmGaussianLowerBounds}. Therefore,
\begin{equation*}
\begin{aligned}
  \bigl|p_t^L - \e_{\Delta_1}^\nu * \cdots * \e_{\Delta_N}^\nu\bigr| &\leq \sum_{j=1}^N |p_{\tau_{j-1}}^L| * \bigl| p_{\Delta_j}^L - \e_{\Delta_j}^\nu\bigr|* |\e_{\Delta_{j+1}}^\nu| * \cdots * |\e_{\Delta_N}^\nu|\\
  &\leq  \sum_{j=1}^N e^{\tau_{j-1}\alpha }p_{\tau_{j-1}}^\Delta * \bigl| p_{\Delta_j}^L - \e_{\Delta_j}^\nu\bigr|*e^{\alpha \Delta_{j+1}}p_{\Delta_{j+1}}^\Delta * \cdots * e^{\alpha \Delta_N}p_{\Delta_N}^\Delta\\
  &\leq e^{\alpha t} \sum_{j=1}^N p_{\tau_{j-1}}^\Delta * \bigl| p_{\Delta_j}^L - \e_{\Delta_j}^\nu\bigr| * p_{t-\tau_j}^\Delta.
\end{aligned}
\end{equation*}
Now, by Thm.~\ref{ThmStrongEstimates} and the Gaussian estimate from below,
\begin{equation*}
\begin{aligned}
\bigl| p_{t}^L(x, y) &- \e_{t}^\nu(x, y)\bigr| \\
 &\leq \left| \Bigl(1 - \chi\bigl(d(x, y)\bigr)\Bigr) p_t^L(x, y)\right| + \left| \chi\bigl(d(x, y)\bigr) \left(p_t^L(x, y) - \e_t(x, y)\sum_{j=0}^\nu t^j \frac{\Phi_j(x, y)}{j!}\right)\right|\\
&\leq e^{t\alpha}\Bigl(1 - \chi\bigl(d(x, y)\bigr)\Bigr) p_t^\Delta(x, y) +  C_1 t^{\nu+1} p_t^\Delta(x, y) 
\end{aligned}
\end{equation*}
Therefore,
\begin{equation*}
\begin{aligned}
\bigl|&p_t^L(x, y) - \bigl(\e_{\Delta_1}^\nu * \cdots * \e_{\Delta_N}^\nu\bigr)(x, y)\bigr| \\
&~\leq e^{\alpha t} \underbrace{\sum_{j=1}^N C_1 \Delta_j^{\nu+1} p_t^\Delta(x, y)}_{(1)} + e^{\alpha t} \underbrace{\sum_{j=1}^N \int_{d(z_0, z_1) \geq R} p_{\tau_{j-1}}^\Delta(x, z_0) p_{\Delta_j}^\Delta(z_0, z_1) p_{t-\tau_j}(z_1, y) \,\dd(z_0, z_1)}_{(2)},
\end{aligned}
\end{equation*}
where $R$ is such that $\chi(r) = 1$ for $0 \leq r \leq R$.
The first term can be estimated by
\begin{equation*}
(1) \leq C_1 |\tau|^\nu \sum_{j=1}^N \Delta_j p_t^\Delta(x, y) = C_1 t|\tau|^\nu p_t^\Delta(x, y).
\end{equation*}
By Lemma~\ref{LemmaLargeOscillations}, whenever $|\tau| \leq \delta t$, the second term can be estimated by
\begin{equation*}
   (2) \leq C_2 \sum_{j=1}^N e^{-\epsilon/\Delta_j} p_t^\Delta (x, y) \leq C_3 e^{-\epsilon^\prime/|\tau|} p_t^\Delta(x, y) \leq C_4 t |\tau|^\nu p_t^\Delta(x, y),
\end{equation*}
with $\epsilon, \epsilon^\prime>0$. This finishes the proof.
\end{proof}

\section{Heat Kernel Asymptotics at the Cut Locus} \label{SectionAsymptoticsFar}

In this section, we use the convolution approximation from Thm.~\ref{ThmConvolutionApproximation} to obtain short-time asymptotic expansions of the heat kernel also in the case that $x, y \in M$ lie in each other's cut locus. As we will see, the form of such an asymptotic expansion depends on the behavior of the energy functional near its critical points on the space of paths between $x$ and $y$.

For an absolutely continuous path $\gamma: [0, 1] \longrightarrow M$, consider the {\em energy functional}
\begin{equation} \label{EnergyFunctional}
  E(\gamma) = \frac{1}{2} \int_0^1 \bigl|\dot{\gamma}(s)\bigr|^2 \dd s.
\end{equation}
Set
\begin{equation*}
  H_{xy}(M) := \bigl\{ \gamma \mid \gamma~\text{is absolutely continuous with}~ E(\gamma) < \infty\bigr\}.
\end{equation*}
This an infinite-dimensional manifold modelled on the Hilbert space $H^1([0, 1], \R^n)$. For details on the manifold structure on $H_{xy}(M)$, see e.g.\ Section~2.3 in \cite{klingenberg}).

Let $\Gamma_{xy}^{\min} \subset H_{xy}(M)$ denote the set of length minimizing geodesics between the points $x, y \in M$. It is well-known that for each $\gamma \in \Gamma_{xy}^{\min}$, we have $E(\gamma) = d(x, y)^2/2$, and conversely, the set $\Gamma_{xy}^{\min}$ is exactly the set of global minima of $E$ on $H_{xy}(M)$. Moreover, $\Gamma_{xy}^{\min}$ is compact in $H_{xy}(M)$ \cite[Prop.~2.4.11]{klingenberg}.

\begin{definition} \label{DefNonDegenerate}
Let $x, y \in M$. We say that $\Gamma_{xy}^{\min}$ is a {\em non-degenerate submanifold}, if it is a submanifold of $H_{xy}(M)$, and if furthermore for each $\gamma \in \Gamma_{xy}^{\min}$, the Hessian of $E$ is non-degenerate when restricted to a complementary subspace to the tangent space $T_\gamma \Gamma_{xy}^{\min}$.
\end{definition}

This is just the well-known Morse-Bott condition on the energy function near the submanifold $\Gamma_{xy}^{\min}$.

\begin{theorem}[Short-time asymptotics, cut locus] \label{ThmEstimatesAwayFromCutLocus}
 Let $M$ be a compact manifold and let $L$ be a self-adjoint Laplace-type \mbox{operator}, acting on sections of a metric vector bundle $\V$ over $M$. For $x, y \in M$, assume that the set $\Gamma_{xy}^{\min}$ is a disjoint union of $k$ non-degenerate submanifolds of dimensions $d_1, \dots, d_k$. Then the heat kernel has the complete asymptotic expansion
 \begin{equation*}
   \frac{p_t^L(x, y)}{\e_t(x, y)} ~\sim~ \sum_{l=1}^k (4\pi t)^{-d_l/2} \sum_{j=0}^\infty t^j \frac{\Phi_{j, l}(x, y)}{j!}
 \end{equation*}
as $t \rightarrow 0$.
\end{theorem}

\begin{remark}  In particular, if $(x, y) \in M \bowtie M$ so that $\Gamma_{xy}^{\min} = \{\gamma\}$ with $\gamma$ the unique minimizing geodesic between $x$ and $y$, then we recover the asymptotic expansion from before, Thm.~\ref{ThmStrongEstimates}.
\end{remark}

\begin{remark}
 The Hessian of the energy at an element $\gamma \in \Gamma_{xy}^{\min}$ can be explicitly calculated and is closely related to the Jacobi equation, see e.g.\ \cite[Section 13]{MilnorMorseTheory}.
\end{remark}

\begin{remark}
Thm.~\ref{ThmEstimatesAwayFromCutLocus} can be generalized to the case that $\Gamma_{xy}^{\min}$ is a {\em degenerate} submanifold of $H_{xy}(M)$. In this case, the explicit form of the asymptotic expansion depends on the type of degeneracy of $E$. In general, it can become quite complicated; for example it may contain logarithmic terms. For a discussion of this, see \cite[ pp.~20-24]{molchanov}.
\end{remark}

\begin{example}
A prototypical example where $\Gamma_{xy}^{\min}$ is a non-degenerate submanifold of dimension greater than zero is when $x$ and $y$ are antipodal points on a sphere. In this case, $\dim \Gamma_{xy}^{\min} = n-1$. For an explicit calculation of $\Phi_0(x, y)$ in this case, see \cite[Example~5.3.3]{hsu}.
\end{example}

The convolution approximation from Thm.~\ref{ThmConvolutionApproximation} is connected to the energy functional as follows. For $x, y \in M$ fixed, set
\begin{equation*}
  M^{(N-1)} := \bigl\{ (x_1, \dots, x_{N-1}) \in M^{N-1} \mid (x_{j-1}, x_j) \in M \bowtie M~\text{for}~j=1, \dots, N\bigr\},
\end{equation*}
with the convention $x_0 := x$, $x_N := y$. For any partition $\tau = \{0 = \tau_0 < \tau_1 < \dots < \tau_N = 1\}$ of the interval $[0, 1]$, the manifold $M^{(N-1)}$ is diffeomorphic to the finite-dimensional submanifold 
\begin{equation*}
  H_{xy;\tau}(M) := \bigl\{ \gamma \in H_{xy}(M) \mid \gamma|_{[\tau_{j-1}, \tau_j]}~ \text{is a unique minimizing geodesic for each}~j\bigr\}
\end{equation*}
of $H_{xy}(M)$ (by the condition that the paths $\gamma$ by unique minimizing, we want to express that we require $(\gamma(\tau_{j-1}), \gamma(\tau_j)) \in M \bowtie M$).
Namely, the evaluation map
\begin{equation*}
 \mathrm{ev}_\tau: H_{xy;\tau}(M) \longrightarrow M^{(N-1)}, ~~~~ \gamma \longmapsto \bigl(\gamma(\tau_1), \dots, \gamma(\tau_{N-1})\bigr)
\end{equation*}
is a  diffeomorphism between the two. For our purpose, it doesn't matter which Riemannian metric (or volume) we put on $H_{xy;\tau}(M)$; for simplicity we take the one that makes $\mathrm{ev}_\tau$ an isometry. 

\begin{lemma} \label{LemmaPathIntegral}
  The heat convolution product from Thm.~\ref{ThmConvolutionApproximation} can be written as an integral over $H_{xy;\tau}(M)$. More specifically, for a partition $\tau= \{ 0 = \tau_0 < \tau_1 < \dots < \tau_N = t\}$, denote by $\tilde{\tau}$ the corresponding partition of the interval $[0, 1]$, given by $\tilde{\tau}_j = \tau_j / t$. Then we have
\begin{equation} \label{PathIntegral}
  \bigl(\e_{\Delta_1 \tau}^\nu * \cdots * \e_{\Delta_N\tau}^\nu\bigr)(x, y) = (4\pi t)^{-nN/2} \int_{H_{xy;\tilde{\tau}}(M)} e^{-E(\gamma)/2t} \,\Upsilon^{\tilde{\tau}, \nu}(t, \gamma) \,\dd \gamma,
\end{equation}
where the integrand $\Upsilon_{\tilde{\tau}, \nu}(t, \gamma)$ is a certain smooth and compactly function on $H_{xy;\tau}(M)$ with values in $\mathrm{Hom}(\V_y, \V_x)$ that depends polynomially on $t$.
\end{lemma}

\begin{proof}
Notice that for the path $\gamma \in H_{xy;\tau}(M)$ with $\gamma(\tau_j) = x_j$, we have
\begin{equation*}
  E(\gamma) = \frac{1}{2}\sum_{j=1}^N \frac{d(x_{j-1}, x_j)^2}{\Delta_j\tau}.
\end{equation*}
Now because the approximate heat kernel $\e_t^\nu$ is supported in $M \bowtie M$ (by choice of the cutoff function present in its definition), the convolution $\e_{\Delta_1 \tau}^\nu * \cdots * \e_{\Delta_N\tau}^\nu$ can be written as an integral over $M^{(N-1)}$,
\begin{equation*}
\begin{aligned}
\bigl(\e_{\Delta_1 \tau}^\nu * \cdots * \e_{\Delta_N\tau}^\nu\bigr)(x, y)&\\
= \int_{M^{(N-1)}} \exp&\left(-\frac{1}{4t}\sum_{j=1}^N d(x_{j-1}, x_j)^2\right) \cdot \\
&\cdot\prod_{j=1}^N \left[ \frac{\chi\bigl(d(x_{j-1}, x_j)\bigr)}{(4\pi \Delta_j\tau)^{n/2}} \sum_{i=1}^\nu (\Delta_j\tau)^j \frac{\Phi_i(x_{j-1}, x_j)}{i!}\right] \,\dd x_1\cdots \dd x_{N-1}.
\end{aligned}
\end{equation*}
We obtain formula \eqref{PathIntegral}, where
\begin{equation} \label{FormulaUpsilonTauNu}
  \Upsilon^{\tilde{\tau}, \nu}(t, \gamma) = \prod_{j=1}^N \left[(\Delta_j\tilde{\tau})^{-n/2} \chi\bigl(d(\gamma(\tilde{\tau}_{j-1}), \gamma(\tilde{\tau}_{j})\bigr)\sum_{i=1}^\nu (t \Delta_j\tilde{\tau})^j \frac{\Phi_i\bigl(\gamma(\tilde{\tau}_{j-1}),\gamma(\tilde{\tau}_{j})\bigr)}{i!}\right].
\end{equation}
This finishes the proof.
\end{proof}

The explicit formula for $\Upsilon^{\tilde{\tau}, \nu}$ is entirely unimportant for our purposes; we only take from it that $\Upsilon^{\tilde{\tau}, \nu}$ is a smooth, compactly supported function on $H_{xy;\tau}(M)$ that depends polynomally on $t$. 

Below, we will always write $\tau$ instead of $\tilde{\tau}$ for a partition of the interval $[0, 1]$.

\begin{proof}[of Thm.~\ref{ThmEstimatesAwayFromCutLocus}]
We will use Laplace's method on the path integral \eqref{PathIntegral}. In order to do this, we have to bring it into the form of Thm.~\ref{ThmAsymptoticExpansion} first, which is achieved by dividing by $\e_t(x, y)$ and setting $\phi(\gamma) := E(\gamma) - d(x, y)/2$. Then by Lemma~\ref{LemmaPathIntegral}, we obtain
\begin{equation*}
  \frac{\bigl(\e_{t\Delta_1 \tau}^\nu * \cdots * \e_{t\Delta_N\tau}^\nu\bigr)(x, y)}{\e_t(x, y)} = (4\pi t)^{-n(N-1)/2} \int_{H_{xy;{\tau}}(M)} e^{-\phi(\gamma)/2t} \,\Upsilon^{{\tau}, \nu}(t, \gamma) \,\dd \gamma,
\end{equation*}
which has the form \eqref{DefTheIntegral} since $\dim(H_{xy;\tau}(M)) = n(N-1)$.

It is clear that whenever the partition $\tau = \{ 0 = \tau_0 < \tau_1 < \dots < \tau_N = 1\}$ is fine enough, we have $\Gamma_{xy}^{\min} \subset H_{xy;\tau}(M)$. By assumption, $\Gamma_{xy}^{\min}$ is the direct sum of non-degenerate submanifolds $\Gamma_1, \dots, \Gamma_k$ of dimension $d_1, \dots, d_k$. Therefore, by Thm.~\ref{ThmAsymptoticExpansion} we obtain the asymptotic expansion
\begin{equation} \label{AsymptoticExpansionInutau}
  \frac{\bigl(\e_{t\Delta_1 \tau}^\nu * \cdots * \e_{t\Delta_N\tau}^\nu\bigr)(x, y)}{\e_t(x, y)} ~\sim~ \sum_{l=1}^k (4\pi t)^{-d_l/2} \sum_{j=0}^\infty t^j  \frac{\Phi_{j, l}^{\tau, \nu}(x, y) }{j!},
\end{equation}
where
\begin{equation} \label{IntegralFormulaPhij}
  \Phi_{j, l}^{\tau, \nu}(x, y) = \sum_{i=0}^j \frac{1}{i!(j-i)!}\int_{\Gamma_l} \frac{P^{j-i}{\Upsilon^{\tau, \nu}}^{(i)}(0, \gamma)}{\det \bigl(\nabla^2 E|_{N_\gamma\Gamma_l}\bigr)^{1/2}}\, \dd \gamma
\end{equation}
for some second order differential operator $P$ on $H_{xy;\tau}(M)$. Here, $\det (\nabla^2 E|_{N_\gamma\Gamma_l})^{1/2}$ denotes the determinant of $\nabla^2 E|_\gamma$, restricted to the normal space $N_\gamma\Gamma_l$ of $T_\gamma\Gamma_l$ in $T_\gamma H_{xy;\tau}(M)$. In particular, if we set $d := \max_{1\leq l \leq k} d_l$, there exists a constant $C_0>0$ such that
\begin{equation} \label{EstimateConvolution}
  \left|\frac{\bigl(\e_{t\Delta_1 \tau}^\nu * \cdots * \e_{t\Delta_N\tau}^\nu\bigr)(x, y)}{\e_t(x, y)}\right| \leq C_0 t^{-d/2}
\end{equation}
for all $0 < t \leq T$.

By Thm.~\ref{ThmConvolutionApproximation}, for each $T>0$ and each $\nu \in \N_0$, there exist constants $C_1, \delta>0$ such that
\begin{equation}\label{PathIntegralEstimateLater}
  \left| \frac{p^L_t(x, y)}{\e_t(x, y)} - \frac{\bigl(\e_{t\Delta_1 \tau}^\nu * \cdots * \e_{t\Delta_N\tau}^\nu\bigr)(x, y)}{\e_t(x, y)}\right| \leq C_1 t^{1+\nu}|\tau|^\nu \frac{p_t^\Delta(x, y)}{\e_t(x, y)},
\end{equation}
for any partition $\tau$ of the interval $[0, 1]$ with $|\tau| \leq \delta$.
By the Gaussian estimate from above (Thm.~\ref{ThmGaussianUpperBounds}) follows $p_t^\Delta (x, y) \leq C_2 t^{-n/2-1} \e_t(x, y)$. Therefore  \eqref{PathIntegralEstimateLater} yields
\begin{equation} \label{PreviousAsymptotics}
  \left| \frac{p_t^L(x, y)}{\e_t(x, y)} -  \frac{\bigl(\e_{t\Delta_1 \tau}^\nu * \cdots * \e_{t\Delta_N\tau}^\nu\bigr)(x, y)}{\e_t(x, y)} \right| \leq C_3 t^{\nu-n/2} |\tau|^\nu.
\end{equation}
Using \eqref{PreviousAsymptotics} and \eqref{EstimateConvolution} for $L=\Delta$, the Laplace-Beltrami operator on $M$, some $\nu \geq n/2 - k/2 -1$ and $|\tau|\leq \delta$, we get
\begin{equation*}
\begin{aligned}
  \frac{p^\Delta_t(x, y)}{\e_t(x, y)} 
  &\leq \left| \frac{p_t^\Delta(x, y)}{\e_t(x, y)} -  \frac{\bigl(\e_{t\Delta_1 \tau}^\nu * \cdots * \e_{t\Delta_N\tau}^\nu\bigr)(x, y)}{\e_t(x, y)} \right| + \left| \frac{\bigl(\e_{\Delta_1 \tau}^\nu * \cdots * \e_{\Delta_N\tau}^\nu\bigr)(x, y)}{\e_t(x, y)}\right|\\
  &\leq C_4 t^{\nu-n/2} |\tau|^\nu + C_5 t^{-d/2} \leq (C_4 \delta^\nu + C_5) t^{-d/2} =: C_6 t^{-d/2}.
\end{aligned}
\end{equation*}
Therefore, \eqref{PathIntegralEstimateLater} improves to
\begin{equation*}
  \left| \frac{p^L_t(x, y)}{\e_t(x, y)} -  \frac{\bigl(\e_{t\Delta_1 \tau}^\nu * \cdots * \e_{t\Delta_N\tau}^\nu\bigr)(x, y)}{\e_t(x, y)}\right| \leq C_1 t^{1+\nu}|\tau|^\nu \cdot C_6 t^{-d/2} \leq C_7 t^{1+\nu - d/2}
\end{equation*}
From this follows that the heat kernel has an asymptotic expansion up to the order $t^{\nu-d/2}$, the coefficients of which must coincide with the asymptotic expansion \eqref{AsymptoticExpansionInutau} of $\e_{t\Delta_1 \tau}^\nu * \cdots * \e_{t\Delta_N\tau}^\nu$ up to that order. Because asymptotic expansions are unique, this also shows that the coefficients $\Phi_{j, l}^{\tau, \nu}(x, y)$ from \eqref{AsymptoticExpansionInutau} must stabilize for $\nu$ large enough and $\tau$ fine enough. More precisely, if $j \leq \nu, \nu^\prime$ and $|\tau|, |\tau^\prime|\leq \delta$, we have 
\begin{equation*}
\Phi_{j, l}^{\tau, \nu}(x, y) = \Phi_{j, l}^{\tau^\prime, \nu^\prime}(x, y).
\end{equation*}
 Therefore
\begin{equation*} 
 \Phi_{j, l}(x, y) := \Phi_{j, l}^{\tau, \nu}(x, y)
\end{equation*}
for any choice of $\nu \geq j$ and $|\tau| \leq  \delta$ is well defined.

Because $\nu$ was arbitrary, we obtain that $p_t^L(x, y)/\e_t(x, y)$ has a complete asymptotic expansion of the claimed form, with the coefficients $\Phi_{j, l}(x, y)$ given by the formula \eqref{IntegralFormulaPhij} for $\nu$ large enough and $|\tau|$ small enough.  
\end{proof}

\appendix

\section{Laplace's method} \label{AppendixLaplace}

Laplace's method is a way to calculate asymptotic expansions as $t \rightarrow 0$ from above for integrals of the form
\begin{equation} \label{DefTheIntegral}
  I(t, a) := (4 \pi t)^{-\dim(\Omega)/2}\int_\Omega e^{-\phi(x)/2t} a(t, x)\, \dd x.
\end{equation}
Here, $t>0$, $\Omega$ is a Riemannian manifold, $\phi \in C^\infty(\Omega)$ is a non-negative function and $a(t, x)$ is smooth and compactly supported with respect to the $x$ variable and depends smoothly on $t$. The following result is very well known, however, it seems that it is nowhere to be found in quite the form needed, so for convenience of the reader, we give a proof in this appendix.

\begin{theorem}[Laplace Expansion] \label{ThmAsymptoticExpansion}
Assume that $\phi$ is non-negative and that $\Gamma := \phi^{-1}(0)$ is a disjoint union of submanifolds $\Gamma_1, \dots, \Gamma_k$ of dimensions $d_1, \dots, d_k$. Suppose that for each $l= 1, \dots, k$, and each $x \in \Gamma_l$, the Hessian $\nabla^2\phi|_x$ is non-degenerate when restricted to the normal space $N_x \Gamma_l$ in $T_x \Omega$. Then $I(t, a)$ has a complete asymptotic expansion as $t$ goes to zero from above. More explicitly, there exists a second order differential operator $P$ such that we have
\begin{equation} \label{Expansion2}
  I(t, a) \sim  \sum_{l=1}^k (4\pi t)^{-d_l/2}\sum_{j=0}^\infty t^j \sum_{i=0}^j \frac{1}{i!(j-i)!}\int_{\Gamma_l} \frac{P^{j-i}a^{(i)}(0, x)}{\det \bigl(\nabla^2 \phi|_{N_x\Gamma_l}\bigr)^{1/2}}\, \dd x
\end{equation}
where $a^{(i)}(0, x)$ denotes the $i$-th derivative of $a$ with respect to $t$ at $t=0$.
\end{theorem}

\begin{remark}
The Laplace expansion of an integral of the form $I(t, a)$ is closely related to the {\em method of stationary phase}, which calculates asymptotic expansions of the integral $t \mapsto I(it, a)$. Laplace's method is simpler in the sense that here, only critical points which are minima contribute to the asymptotic expansion, while for integrals with imaginary exponent, all critical points contribute. Compare e.g.\ \cite{ArnoldStationaryPhase} or \cite[Section~1.2]{Duistermaat}.
\end{remark}

\begin{lemma} \label{LemmaExponentialDecay}
Under the assumptions of Thm.~\ref{ThmAsymptoticExpansion}, suppose that $a(t, x) = 0$ for all $x$ in a neighborhood of $\Gamma$ and all $0\leq t \leq \delta$, for some $\delta>0$. Then there exist constants $T, C, \varepsilon>0$ such that for all $t \leq T$, we have $I(t, a) \leq C e^{-\varepsilon/t}$.
\end{lemma}

\begin{proof}
  Let $N := \dim(\Omega)$. Set
\begin{equation} \label{SupportA}
  A := \text{closure of}~~\bigcup_{0 \leq t \leq \delta} \mathrm{supp}\, a(t, -)
\end{equation}
(which is compact) and set
\begin{equation*}
  \varepsilon^\prime := \min_{x \in A} \phi(x).
\end{equation*}  
Notice that $\varepsilon^\prime>0$ because $A \cap \Gamma = \emptyset$. Therefore,
  \begin{equation*}
    I(t, a) \leq (4 \pi t)^{-N/2} e^{-\varepsilon^\prime/2t} \int_\Omega a(t, x) \dd x \leq (4\pi t)^{-N/2} e^{-\varepsilon^\prime/2t} \|a(t, -)\|_{L^1} \leq C e^{-\varepsilon/t},
  \end{equation*}
  if we choose $0 < \varepsilon < \varepsilon^\prime$ and $C>0$ appropriately.
\end{proof}

\begin{proof}[of Thm.~\ref{ThmAsymptoticExpansion}]
We may write the integral over $\Omega$ as a sum of integrals over open subsets $\Omega_1, \dots, \Omega_k$ such that the union of the $\Omega_l$ is dense in $\Omega$, and such that $\Gamma_l \subset \Omega_l$ for each $l=1, \dots, k$. The asymptotic expansion of the integral over $\Omega$ will then the be sum of the asymptotic expansions of the integrals over the manifolds $\Omega_l$. Therefore, we may assume that $k=1$, i.e.\ $\Gamma$ is a non-degenerate submanifold of dimension $d$.

Let $N := \dim(\Omega)$ and let $A$ as in \eqref{SupportA}. Since $A$ is compact, we may without loss of generality assume that also $\Omega$ and hence $\Gamma$ is compact. Otherwise embed some open neighborhood of $A$ isometrically into a compact manifold $\Omega^\prime$, transplant $\phi$ and $a$ there and replace $\Omega$ by $\Omega^\prime$ in the definition of $I(t, a)$. This does not alter the value of $I(t, a)$. 

Let $N\Gamma \subseteq T\Omega$ be the normal bundle of $\Gamma$. Then there is an open neighborhood $V$ of the zero section in $N\Gamma$ and an open neighborhood $U$ of $\Gamma$ in $\Omega$ together with a diffeomorphism $\kappa: V \longrightarrow U$ such that
\begin{equation*}
  \bigl(\phi \circ \kappa\bigr) (x, v) = \nabla^2\phi|_x[v, v], ~~~~~~~ (x, v) \in V.
\end{equation*}
This can be proved using the implicit function theorem, compare e.g.\ Lemma~1.2.2 in \cite{Duistermaat}. Clearly, we have $d \kappa|_{(x, 0)} = \id_x$.

Furthermore, we may assume that $A \subset U$. Namely otherwise, we can choose a cutoff function $\chi \in C^\infty_c(U)$ that is equal to one on a neighborhood of $\Gamma$ and split  $I(t, a) = I(t, \chi a) + I(t, (1-\chi)a)$, where the second summand does not contribute to the asymptotic expansion because of Lemma~\ref{LemmaExponentialDecay}. 

We now may use the transformation formula to obtain
\begin{equation} \label{EquationFranzi}
\begin{aligned}
  I(t, a) &= (4 \pi t)^{-N/2}\int_U e^{-\phi(x)/2t} a(t, x)\, \dd x\\
  &= (4\pi t)^{-N/2} \int_\Gamma \int_{V_x} e^{-\<v, Q(x)v\>/4t} a\bigl(t,  \kappa(x, v)\bigr) \bigl|\det\bigl(d \kappa|_{(x,v)}\bigr)\bigr| \dd v \dd x,
\end{aligned}
\end{equation}
where we wrote $Q(x) := \nabla^2\phi|_{N_x\Gamma}$ and $V_x := V \cap N_x \Gamma$. It is well known that for any $(N-d)$-dimensional Euclidean vector space $W$, any positive definite endomorphism $Q$ of $W$ and any continuous function $f = f(t, x)$ on $\R \times W$ which is bounded in the $x$ variable and depends smoothly on $t$, one has
\begin{equation*}
  \lim_{t \rightarrow 0}(4\pi t)^{-(N-d)/2}\int_W e^{-\<v, Qv\>/4t} f(t, v) \dd v = \det(Q)^{-1/2} f(0, 0).
\end{equation*}
Furthermore, for all $t$, we have
\begin{equation*}
  \left|(4\pi t)^{-(N-d)/2}\int_W e^{-\<v, Qv\>/4t} f(t, v) \dd v\right| \leq \|f(t, -)\|_\infty.
\end{equation*}
Therefore since $\Gamma$ is compact, we may exchange integration over $\Gamma$ and the limit $t \rightarrow 0$ in \eqref{EquationFranzi} to conclude
\begin{equation} \label{ShortTimeLimit}
  \lim_{t \rightarrow 0} (4 \pi t)^{d/2} I(t, a) = \int_\Gamma \frac{a\bigl(0, \kappa(x, 0)\bigr)}{\det\bigl(Q(x)\bigr)^{1/2}} \bigl|\det\bigl(d \kappa|_{(x, 0)}\bigr)\bigr| \dd x = \int_\Gamma \frac{a(0, x)}{\det\bigl(\nabla^2\phi|_{N_x \Gamma}\bigr)^{1/2}} \dd x
\end{equation}
Now on the vector spaces $N_x \Gamma$, define the $Q$-Laplacian $\Delta_Q$ by the formula 
\begin{equation*}
\Delta_{Q} f(v) = -\<Q(x)^{-1}, D^2f|_v\>.
\end{equation*}
This patches together to a smooth differential operator on $N\Gamma$ satisfying
\begin{equation*}
  \Bigl(\frac{\partial}{\partial t} + \Delta_{Q}\Bigr)\bigl\{(4\pi t)^{-(N-d)/2} e^{-\<v, Q(x)v\>/4t}\bigr\} = 0.
\end{equation*}
Therefore, integrating by parts, we obtain
\begin{equation*}
\begin{aligned}
  \frac{\partial}{\partial t} \bigl\{(4 \pi t)^{d/2}& I\bigl(t, a\bigr)\bigr\} - (4 \pi t)^{d/2} I\bigl(t, \dot{a}\bigr)\\
  &= -(4\pi t)^{-(N-d)/2}\int_\Gamma \int_{V_x} e^{-\< v, Q(x)v\>/4t} \Delta_Q \Bigl\{ a\bigl(t,\kappa(x, v)\bigr) \bigl|\det\bigl(d \kappa|_{(x,v)}\bigr)\bigr|\Bigr\} \dd v\dd x\\
  &= (4\pi t)^{-(N-d)/2}\int_U e^{-\phi(x) /2t} P a(t, x) \dd x = (4 \pi t)^{d/2}I(t, Pa),
\end{aligned}
\end{equation*}
where for $f \in C^\infty(U)$, we set
\begin{equation*}
  (Pf)(y) = -\Delta_Q\bigl\{f(v) \bigl|\det\bigl(d \kappa|_{(x, v)}\bigr)\bigr|\bigr\}\bigl|_{(x,v) = \kappa^{-1}(y)}\bigl|\det\bigl(d \kappa^{-1}|_{y}\bigr)\bigr|,
\end{equation*}
so that $P$ is some second-order differential operator. Let $J(t, a) := (4 \pi t)^{d/2}I(t, a)$. Then by Taylor's formula and the Leibnitz rule, for all $\varepsilon >0$ and $\nu \in \N$,
\begin{align*}
 J(t, a) &= \sum_{j=0}^\nu \frac{1}{j!} \frac{\partial^j}{\partial \varepsilon^j} \bigl\{J(\varepsilon, a)\bigr\} (t-\varepsilon)^j + \int_\varepsilon^t \frac{(t-s)^\nu}{\nu!} \frac{\partial^{\nu+1}}{\partial s^{\nu+1}} \bigl\{J(s, a)\bigr\} \dd s \\
 &= \sum_{j=0}^\nu \frac{1}{j!}\sum_{i=0}^j \binom{j}{i} J\bigl(\varepsilon, P^{j-i}a^{(i)}\bigr) (t-\varepsilon)^j + R^\nu(\varepsilon, t),
\end{align*}
where
\begin{equation}
  R^\nu(\varepsilon, t) = \sum_{i=0}^{\nu+1} \binom{\nu+1}{i} \int_\varepsilon^t \frac{(t-s)^\nu}{\nu!}  J\bigl(s, P^{\nu+1-i}a^{(i)}\bigr) \dd s.
\end{equation}
Because of \eqref{ShortTimeLimit}, we may take the limit $\varepsilon \rightarrow 0$ to obtain
\begin{equation*}
\lim_{\varepsilon \rightarrow 0} J\bigl(\varepsilon, P^{j-i}a^{(i)}\bigr) = \int_\Gamma \frac{P^{j-i}a^{(i)}(0, x)}{\det \bigl(\nabla^2 \phi|_{N_x \Gamma}\bigr)^{1/2}} \dd x.
\end{equation*}
Therefore,
\begin{equation*}
  J(t, a) = \sum_{j=0}^\nu t^j \sum_{i=0}^j \frac{1}{(j-i)!i!}\int_\Gamma \frac{P^{j-i}a^{(i)}(0, x)}{\,\det \bigl(\nabla^2 \phi|_{N_x \Gamma}\bigr)^{1/2}} \dd x + R^\nu(0, t),
\end{equation*}
for any $\nu \in \N_0$, where the remainder term is of order $t^{\nu+1}$. This finishes the proof.
\end{proof}

    \bibliography{Literatur}

\begin{thebibliography}{BGM71}

\bibitem[Arn73]{ArnoldStationaryPhase}
V.~I. Arnol'd.
\newblock Remarks on the method of stationary phase and on the {C}oxeter
  numbers.
\newblock {\em Uspehi Mat. Nauk}, 28(5(173)):17--44, 1973.

\bibitem[Aze84]{azencott}
Robert Azencott.
\newblock Densit\'e des diffusions en temps petit: d\'eveloppements
  asymptotiques. {I}.
\newblock In {\em Seminar on probability, {XVIII}}, volume 1059 of {\em Lecture
  Notes in Math.}, pages 402--498. Springer, Berlin, 1984.

\bibitem[BA88]{benarous}
G.~Ben~Arous.
\newblock D\'eveloppement asymptotique du noyau de la chaleur hypoelliptique
  hors du cut-locus.
\newblock {\em Ann. Sci. \'Ecole Norm. Sup. (4)}, 21(3):307--331, 1988.

\bibitem[BBN12]{BarilariBoscainNeel}
Davide Barilari, Ugo Boscain, and Robert~W. Neel.
\newblock Small-time heat kernel asymptotics at the sub-{R}iemannian cut locus.
\newblock {\em J. Differential Geom.}, 92(3):373--416, 2012.

\bibitem[BGM71]{BergerGauduchonMazet}
Marcel Berger, Paul Gauduchon, and Edmond Mazet.
\newblock {\em Le spectre d'une vari\'et\'e riemannienne}.
\newblock Lecture Notes in Mathematics, Vol. 194. Springer-Verlag, Berlin-New
  York, 1971.

\bibitem[BGP07]{bgp}
Christian B\"ar, Nicolas Ginoux, and Frank Pf\"affle.
\newblock {\em Wave equations on {L}orentzian manifolds and quantization}.
\newblock ESI Lectures in Mathematics and Physics. European Mathematical
  Society (EMS), Z\"urich, 2007.

\bibitem[BGV04]{bgv}
Nicole Berline, Ezra Getzler, and Mich\`ele Vergne.
\newblock {\em Heat kernels and {D}irac operators}.
\newblock Grundlehren Text Editions. Springer-Verlag, Berlin, 2004.
\newblock Corrected reprint of the 1992 original.

\bibitem[BS91]{SchroedingerEquation}
F.~A. Berezin and M.~A. Shubin.
\newblock {\em The {S}chr\"odinger equation}, volume~66 of {\em Mathematics and
  its Applications (Soviet Series)}.
\newblock Kluwer Academic Publishers Group, Dordrecht, 1991.
\newblock Translated from the 1983 Russian edition by Yu. Rajabov, D. A. Le\u\i
  tes and N. A. Sakharova and revised by Shubin, With contributions by G. L.
  Litvinov and Le\u\i tes.

\bibitem[BW15]{BaerTagneWafo}
Christian B\"ar and Roger~Tagne Wafo.
\newblock Initial value problems for wave equations on manifolds.
\newblock {\em Math. Phys. Anal. Geom.}, 18(1):Art. 7, 29, 2015.

\bibitem[Cha84]{ChavelEigenvalues}
Isaac Chavel.
\newblock {\em Eigenvalues in {R}iemannian geometry}, volume 115 of {\em Pure
  and Applied Mathematics}.
\newblock Academic Press, Inc., Orlando, FL, 1984.
\newblock Including a chapter by Burton Randol, With an appendix by Jozef
  Dodziuk.

\bibitem[CS08]{coulhonsikora}
Thierry Coulhon and Adam Sikora.
\newblock Gaussian heat kernel upper bounds via the {P}hragm\'en-{L}indel\"of
  theorem.
\newblock {\em Proc. Lond. Math. Soc. (3)}, 96(2):507--544, 2008.

\bibitem[DP89]{daviespang}
E.~B. Davies and M.~M.~H. Pang.
\newblock Sharp heat kernel bounds for some {L}aplace operators.
\newblock {\em Quart. J. Math. Oxford Ser. (2)}, 40(159):281--290, 1989.

\bibitem[Dui11]{Duistermaat}
J.~J. Duistermaat.
\newblock {\em Fourier integral operators}.
\newblock Modern Birkh\"auser Classics. Birkh\"auser/Springer, New York, 2011.
\newblock Reprint of the 1996 edition [MR1362544], based on the original
  lecture notes published in 1973 [MR0451313].

\bibitem[FS08]{FineSawin}
Dana~S. Fine and Stephen~F. Sawin.
\newblock A rigorous path integral for supersymmetic quantum mechanics and the
  heat kernel.
\newblock {\em Comm. Math. Phys.}, 284(1):79--91, 2008.

\bibitem[FS14]{FineSawin2}
Dana~S. Fine and Stephen Sawin.
\newblock Short-time asymptotics of a rigorous path integral for {$N=1$}
  supersymmetric quantum mechanics on a {R}iemannian manifold.
\newblock {\em J. Math. Phys.}, 55(6):062104, 25, 2014.

\bibitem[FS17]{FineSawin3}
Dana~S. Fine and Stephen Sawin.
\newblock Path integrals, supersymmetric quantum mechanics, and the
  {A}tiyah-{S}inger index theorem for twisted {D}irac.
\newblock {\em J. Math. Phys.}, 58(1):012102, 30, 2017.

\bibitem[Gre70]{GreinerHeatEquation}
Peter Greiner.
\newblock An asymptotic expansion for the heat equation.
\newblock In {\em Global {A}nalysis ({P}roc. {S}ympos. {P}ure {M}ath., {V}ol.
  {XVI}, {B}erkeley, {C}alif., 1968)}, pages 133--135. Amer. Math. Soc.,
  Providence, R.I., 1970.

\bibitem[Gri09]{grigoryan}
Alexander Grigor'yan.
\newblock {\em Heat kernel and analysis on manifolds}, volume~47 of {\em AMS/IP
  Studies in Advanced Mathematics}.
\newblock American Mathematical Society, Providence, RI; International Press,
  Boston, MA, 2009.

\bibitem[HSU80]{HessSchraderUhlenbrock}
H.~Hess, R.~Schrader, and D.~A. Uhlenbrock.
\newblock Kato's inequality and the spectral distribution of {L}aplacians on
  compact {R}iemannian manifolds.
\newblock {\em J. Differential Geom.}, 15(1):27--37 (1981), 1980.

\bibitem[Hsu02]{hsu}
Elton~P. Hsu.
\newblock {\em Stochastic analysis on manifolds}, volume~38 of {\em Graduate
  Studies in Mathematics}.
\newblock American Mathematical Society, Providence, RI, 2002.

\bibitem[IT17]{InahamaTaniguchi}
Yuzuru Inahama and Setsuo Taniguchi.
\newblock Short time full asymptotic expansion of hypoelliptic heat kernel at
  the cut locus.
\newblock {\em Forum Math. Sigma}, 5:e16, 74, 2017.

\bibitem[Kan77]{Kannai}
Y.~Kannai.
\newblock Off diagonal short time asymptotics for fundamental solutions of
  diffusion equations.
\newblock {\em Commun. Partial Differ. Equations}, 2(8):781--830, 1977.

\bibitem[Kli95]{klingenberg}
Wilhelm P.~A. Klingenberg.
\newblock {\em Riemannian geometry}, volume~1 of {\em De Gruyter Studies in
  Mathematics}.
\newblock Walter de Gruyter \& Co., Berlin, second edition, 1995.

\bibitem[Mel93]{MelroseAPS}
Richard~B. Melrose.
\newblock {\em The {A}tiyah-{P}atodi-{S}inger index theorem}, volume~4 of {\em
  Research Notes in Mathematics}.
\newblock A K Peters, Ltd., Wellesley, MA, 1993.

\bibitem[Mil63]{MilnorMorseTheory}
J.~Milnor.
\newblock {\em Morse theory}.
\newblock Based on lecture notes by M. Spivak and R. Wells. Annals of
  Mathematics Studies, No. 51. Princeton University Press, Princeton, N.J.,
  1963.

\bibitem[Mol75]{molchanov}
S.~A. Molchanov.
\newblock Diffusion processes, and {R}iemannian geometry.
\newblock {\em Uspehi Mat. Nauk}, 30(1(181)):3--59, 1975.

\bibitem[Mor99]{moretti}
Valter Moretti.
\newblock Proof of the symmetry of the off-diagonal heat-kernel and
  {H}adamard's expansion coefficients in general {$C^\infty$} {R}iemannian
  manifolds.
\newblock {\em Comm. Math. Phys.}, 208(2):283--308, 1999.

\bibitem[Mor00]{moretti2}
Valter Moretti.
\newblock Proof of the symmetry of the off-diagonal
  {H}adamard/{S}eeley-{D}e{W}itt's coefficients in {$C^\infty$} {L}orentzian
  manifolds by a ``local {W}ick rotation''.
\newblock {\em Comm. Math. Phys.}, 212(1):165--189, 2000.

\bibitem[NS04]{NeelStrook}
Robert Neel and Daniel Stroock.
\newblock Analysis of the cut locus via the heat kernel.
\newblock In {\em Surveys in differential geometry. {V}ol. {IX}}, volume~9 of
  {\em Surv. Differ. Geom.}, pages 337--349. Int. Press, Somerville, MA, 2004.

\bibitem[Roe98]{roe98}
John Roe.
\newblock {\em Elliptic operators, topology and asymptotic methods}, volume 395
  of {\em Pitman Research Notes in Mathematics Series}.
\newblock Longman, Harlow, second edition, 1998.

\bibitem[Ros97]{Rosenberg}
Steven Rosenberg.
\newblock {\em The {L}aplacian on a {R}iemannian manifold}, volume~31 of {\em
  London Mathematical Society Student Texts}.
\newblock Cambridge University Press, Cambridge, 1997.
\newblock An introduction to analysis on manifolds.

\bibitem[Sim00]{SimonFeynmanKac}
Barry Simon.
\newblock A {F}eynman-{K}ac formula for unbounded semigroups.
\newblock In {\em Stochastic processes, physics and geometry: new interplays,
  {I} ({L}eipzig, 1999)}, volume~28 of {\em CMS Conf. Proc.}, pages 317--321.
  Amer. Math. Soc., Providence, RI, 2000.

\bibitem[Str83]{Strichartz}
Robert~S. Strichartz.
\newblock Analysis of the {L}aplacian on the complete {R}iemannian manifold.
\newblock {\em J. Funct. Anal.}, 52(1):48--79, 1983.

\bibitem[Tay11]{TaylorTwo}
Michael~E. Taylor.
\newblock {\em Partial differential equations {II}. {Q}ualitative studies of
  linear equations}, volume 116 of {\em Applied Mathematical Sciences}.
\newblock Springer, New York, second edition, 2011.

\bibitem[Wer00]{werner95}
Dirk Werner.
\newblock {\em Funktionalanalysis}.
\newblock Springer-Verlag, Berlin, extended edition, 2000.

\bibitem[Wol73]{WolfDirac}
Joseph~A. Wolf.
\newblock Essential self-adjointness for the {D}irac operator and its square.
\newblock {\em Indiana Univ. Math. J.}, 22:611--640, 1972/73.

\end{thebibliography}

\end{document}